\documentclass[11pt,a4paper]{amsart} 
\usepackage{amssymb,amscd,amsmath, young,mathrsfs}
\usepackage{mathrsfs, mathdots} 
\usepackage{float}
\usepackage[usenames]{color}
\usepackage{soul}
\usepackage{longtable}

\theoremstyle{plain}
\newtheorem{thm}{Theorem}[section]
\newtheorem{lem}[thm]{Lemma}
\newtheorem{pro}[thm]{Proposition}

\newtheorem{lemma}[thm]{Lemma}

\theoremstyle{remark}
\newtheorem{rem}[thm]{Remark}

\newtheorem*{question*}{Question}

\newtheorem*{acknowledgements}{Acknowledgements}

\newcommand{\N}{\mathbb{N}}
\newcommand{\Z}{\mathbb{Z}}
\newcommand{\Q}{\mathbb{Q}}
\newcommand{\F}{\mathbb{F}}
\newcommand{\C}{\mathbb{C}}

\newcommand{\R}{\mathbb{R}}

\newcommand{\p}{\mathfrak{p}}

\renewcommand{\epsilon}{\varepsilon}

\renewcommand{\phi}{\varphi}

\DeclareMathOperator{\tr}{tr}

\DeclareMathOperator{\arcosh}{arcosh}

\def \p {\ensuremath{\mathfrak{p}}}

\author{Michael M.~Schein} \address{Department of Mathematics,
  Bar-Ilan University, Ramat Gan 5290002,
  Israel}
 \email{mschein@math.biu.ac.il}
\author{Amir Shoan} \address{Department of Mathematics,
  Bar-Ilan University, Ramat Gan 5290002,
  Israel}
\email{shoanam@macs.biu.ac.il}

\thanks{The first author was supported by the Germany-Israel Foundation for Scientific Research and Development under grant 1246/2014 during part of this work.}

\begin{document}

 \title[Systolic length of triangular modular curves]{Systolic length of triangular modular curves} 
 
 \maketitle
 \begin{abstract}
We present a method for computing upper bounds on the systolic length of certain Riemann surfaces uniformized by congruence subgroups of hyperbolic triangle groups, admitting congruence Hurwitz curves as a special case.  The uniformizing group is realized as a Fuchsian group and a convenient finite generating set is computed.  The upper bound is derived from the traces of the generators.  Some explicit computations, including ones for non-arithmetic surfaces, are given.  We apply a result of Cosac and D\'{o}ria to show that the systolic length grows logarithmically with respect to the genus.
\end{abstract}
\section{Introduction} \label{sec:intro}
\subsection{Overview}
This article presents a method for using computations with quaternion algebras over totally real fields to determine explicit upper bounds for the systolic lengths of certain Riemann surfaces, namely the Galois-Belyi curves constructed by Clark and Voight~\cite{CV/19}.  These are also referred to as triangular modular curves.  The most celebrated members of this family are the congruence Hurwitz curves.  We generalize the treatment in~\cite{KKSV/16} of a family of Riemann surfaces related to the Bolza surface and extend the methods of~\cite{KKSV/16} to some non-arithmetic surfaces.

Let $X$ be a compact Riemann surface.  Our main object of interest, the systolic length of $X$, is denoted $\mathrm{sys}(X)$ and is the minimal length of a non-contractible closed geodesic, or systole, of $X$.  
Let $\mathcal{H}$ be the upper half plane with the usual hyperbolic metric.  The group $\mathrm{SL}_2(\R)$, via its quotient $\mathrm{PSL}_2(\R)$, acts on $\mathcal{H}$ by M\"{o}bius transformations.  Every hyperbolic Riemann surface arises as a quotient $X_\Gamma = \Gamma \backslash \mathcal{H}$, where $\Gamma \leq \mathrm{PSL}_2(\R)$ is a discrete subgroup (i.e.~a Fuchsian group) acting on $\mathcal{H}$ without fixed points.  As we will see in Section~\ref{sec:traces} below, bounding the systolic length $\mathrm{sys}(X_\Gamma)$ amounts to finding an upper bound for the minimal trace of a non-trivial element of $\Gamma$.

If $\Gamma$ is a non-arithmetic Fuchsian group, there is generally no natural construction of congruence subgroups of $\Gamma$.  However, if $\Gamma = \overline{\Delta}(a,b,c) = \langle x, y \rangle / (x^a = y^b = (xy)^c = 1)$ is a hyperbolic triangle group, then Clark and Voight~\cite{CV/19} define congruence subgroups $\Gamma(I) \trianglelefteq \Gamma$ of finite index, where $I$ runs over ideals of a  Dedekind domain associated to $\Gamma$, and describe the quotient $\Gamma / \Gamma(I)$.  If $\Gamma = \overline{\Delta}(2,3,7)$, then the Riemann surfaces $X_{\Gamma(I)}$ are congruence Hurwitz curves.  For all but finitely many triples $(a,b,c)$, the group $\overline{\Delta}(a,b,c)$ is non-arithmetic.

To estimate $\mathrm{sys}(X_{\Gamma(I)})$, where $\Gamma$ is a hyperbolic triangle group, we compute a finite set of Schreier generators of $\Gamma(I)$.  We then embed $\Gamma(I)$ into $\mathrm{PSL}_2(\R)$ explicitly and consider the minimum among the traces of the images of the Schreier generators.  

A brief overview of the paper follows.
Section~\ref{sec:traces} describes the connection between the systolic length of $X_\Gamma$ and the traces of elements of $\Gamma$.
Section~\ref{sec:triangle.groups} is an exposition of work by previous authors, whose aim is to present the necessary background material in a form useful for our applications.  Some properties of triangle groups and of their congruence subgroups are recalled.  Proposition~\ref{pro:alpha.beta.fulldelta}, which restates a result of Takeuchi~\cite{Takeuchi/77JFS}, provides an embedding $\Gamma \hookrightarrow \mathcal{O}^1 / \{ \pm 1 \}$, where $\mathcal{O}^1$ is the group of elements of reduced norm one in an explicit order $\mathcal{O}$ of an explicit quaternion algebra over a totally real field.  The explicit quaternion order of Proposition~\ref{pro:presentation.a} realizes $\Gamma$ as a Fuchsian group when it is arithmetic.  This construction generalizes ones used, in special cases, in~\cite{KKSV/16, KSV/11} and is useful to us in non-arithmetic cases as well.  Proposition~\ref{pro:counting.normal.subgroups}, using Macbeath's classification~\cite{Macbeath/69} of subgroups of $\mathrm{PSL}_2(\F_q)$ and work of Clark and Voight~\cite{CV/19}, bounds the number of normal subgroups of $\Gamma$ admitting certain finite quotients.  These bounds simplify the computations of Section~\ref{sec:computations} in some cases.  
Section~\ref{sec:semiarithmetic} uses recent work of Cosac and D\'{o}ria~\cite{CosacDoria/20} to prove (Theorem~\ref{thm:semiarithmetic}) that $\mathrm{sys}(X_{\Gamma(I)})$ grows at least logarithmically with respect to the genus of $X_{\Gamma(I)}$ when at least two of $a,b,c$ are odd.  
Section~\ref{sec:computations} presents our method of computation and some examples.  Directions to a Github repository of our code are provided.  In particular, we bound the systolic lengths of curves associated to congruence subgroups for ideals of small norm in the arithmetic triangle groups $\overline{\Delta}(2,3,7)$, $\overline{\Delta}(2,3,8)$, $\overline{\Delta}(2,3,12)$, $\overline{\Delta}(2,7,7)$, as well as in the non-arithmetic $\overline{\Delta}(3,3,10)$.  The first two of these examples were previously considered by Katz, Schaps, and Vishne~\cite{KSV/07, KSV/11} and by Katz, Katz, the first author, and Vishne~\cite{KKSV/16}, respectively; we recover constructions from those papers.  Our upper bound is equal to the exact value of $\mathrm{sys}(X_{\Gamma(I)})$ in all cases, among the examples considered, where this value is known.

\subsection{Traces and geodesic lengths} \label{sec:traces}
If $\gamma \in \mathrm{PSL}_2 (\R)$, then the trace $\tr \gamma$ is defined up to sign, so the absolute value $| \tr \gamma |$ makes sense; the same is true for the eigenvalues of $\gamma$.  If $ | \tr \gamma |  > 2$, then $\gamma$ is called hyperbolic; in this case, a simple calculation shows that $\gamma$ has no fixed points in the interior of $\mathcal{H}$ and two distinct fixed points on the boundary of $\mathcal{H}$.  If the Fuchsian group $\Gamma \leq \mathrm{PSL}_2(\R)$ is torsion-free, then there is a bijection between conjugacy classes of hyperbolic elements of $\Gamma$ and closed geodesics of $X_\Gamma$.  Indeed, if $\gamma \in \Gamma$ is hyperbolic, the unique geodesic of $\mathcal{H}$ connecting the two fixed points of $\gamma$ is called the axis of $\gamma$; clearly it is preserved by the action of $\gamma$.  Fix a point $P$ on the axis of $\gamma$.  The geodesic of $\mathcal{H}$ passing through $P$ and $\gamma P$ descends to a closed geodesic $\delta_\gamma$ of $X_\Gamma$.  One checks that this construction is independent of the choice of $P$ and depends only on the conjugacy class of $\gamma$.  Clearly, any closed geodesic of $X_\Gamma$ arises in this way.  Moreover, the length $\ell(\delta_\gamma)$ of the geodesic $\delta_\gamma$ satisfies $e^{\ell(\delta_\gamma)/2} = | \lambda_\gamma |$, where $\lambda_\gamma$ is the unique eigenvalue of $\gamma$ such that $ | \lambda_\gamma | > 1$.  It follows that
$$ 2 \cosh \frac{\ell(\delta_\gamma)}{2} = e^{\ell(\delta_\gamma)/2} + e^{- \ell (\delta_\gamma)/2} = | \lambda_\gamma | + | \lambda_\gamma |^{-1} = | \tr \gamma |.$$
Hence, letting $\mathrm{Hyp}(\Gamma)$ denote the set of hyperbolic elements of $\Gamma$, we have
\begin{equation*}
\mathrm{sys}(X_\Gamma) = \min_{\gamma \in \mathrm{Hyp}(\Gamma)} \ell(\delta_\gamma) = \min_{\gamma \in \mathrm{Hyp}(\Gamma)} 2 \arcosh \frac{| \tr \gamma |}{2}.
\end{equation*}
Bounding the systolic length of $X_\Gamma$ thus amounts to bounding the traces of hyperbolic elements of the Fuchsian group $\Gamma$.  In particular, if $\Gamma^\prime \leq \Gamma$, then $\mathrm{sys}(X_{\Gamma^\prime}) \geq \mathrm{sys}(X_\Gamma)$.

\section{Triangle groups} \label{sec:triangle.groups}
\subsection{Definition} 
Suppose that $a,b,c \in \Z \cup \{ \infty \}$ satisfy $2 \leq a,b,c$.  The triangle group $\overline{\Delta}(a,b,c)$ is given by the presentation
\begin{equation} \label{equ:triangle.group}
\overline{\Delta}(a,b,c) = \langle x, y \rangle / (x^a = y^b = (xy)^c = 1).
\end{equation}
The extended triangle group $\Delta(a,b,c)$ is defined as
\begin{equation} \label{equ:delta.presentation}
\Delta(a,b,c) = \langle x, y, z, -1 \rangle / \left({x^a = y^b = z^c = xyz = -1, (-1)^2 = 1 \atop -1 \in Z(\Delta(a,b,c))}\right).
\end{equation}
Any permutation of $a,b,c$ in either presentation gives rise to an isomorphic group, so we make the convention that $2 \leq a \leq b \leq c$.
The triple $(a,b,c)$ is classified as spherical, Euclidean, or hyperbolic if the quantity $\frac{1}{a} + \frac{1}{b} + \frac{1}{c}$ is greater than, equal to, or less than $1$, respectively.  If $(a,b,c)$ is a hyperbolic triple, then it has been known at least since the 1930's~\cite[\S1]{Petersson/37} that there is an embedding $\iota: \Delta(a,b,c) \hookrightarrow \mathrm{SL}_2(\R)$ determined by
\begin{eqnarray} \label{equ:iota.embedding}
\iota (x) & = & \left( \begin{array}{cc} \cos \frac{\pi}{a} & \sin \frac{\pi}{a} \\ - \sin \frac{\pi}{a} & \cos \frac{\pi}{a} \end{array} \right),  \\ \nonumber
\iota (y) & = & \left( \begin{array}{cc} 1 & 0 \\ 0 & t \end{array} \right)^{-1} 
\left( \begin{array}{cc} \cos \frac{\pi}{b} & \sin \frac{\pi}{b} \\ - \sin \frac{\pi}{b} & \cos \frac{\pi}{b} \end{array} \right)
\left( \begin{array}{cc} 1 & 0 \\ 0 & t \end{array} \right),
\end{eqnarray}
where $t$ is a root of the quadratic equation 
$$ t^2 - 2 \left( \frac{\cos \frac{\pi}{a} \cos \frac{\pi}{b} + \cos \frac{\pi}{c}}{\sin \frac{\pi}{a} \sin \frac{\pi}{b}} \right) t + 1 = 0;$$
the roots are real when $(a,b,c)$ is hyperbolic.  The induced embedding $\overline{\Delta}(a,b,c) \hookrightarrow \mathrm{PSL}_2(\R)$ realizes $\overline{\Delta}(a,b,c)$ as a Fuchsian group.  If $b < \infty$, then Takeuchi~\cite[Proposition~1]{Takeuchi/77} showed that any embedding of $\Delta(a,b,c)$ into $\mathrm{SL}_2(\R)$ is conjugate to $\iota$.  In particular, if $\Gamma \leq \mathrm{SL}_2(\R)$ satisfies $\Gamma \simeq \Delta(a,b,c)$, then $\mathrm{sys}(X_\Gamma)$ depends only on the triple $(a,b,c)$.

Consider the full triangle group $\widetilde{\Delta}(a,b,c)$ generated by reflections across the sides of a triangle with angles $\frac{\pi}{a}$, $\frac{\pi}{b}$, $\frac{\pi}{c}$.  It is well-known to have the presentation
\begin{equation} \label{equ:full.triangle.group}
\widetilde{\Delta}(a,b,c) = \langle \widetilde{x}, \widetilde{y}, \widetilde{z} \rangle / (\widetilde{x}^2 = \widetilde{y}^2 = \widetilde{z}^2 = (\widetilde{x} \widetilde{y})^a = (\widetilde{y} \widetilde{z})^b = (\widetilde{z} \widetilde{x})^c = 1 ).
\end{equation}
It is easy to see that $\overline{\Delta}(a,b,c)$ embeds in $\widetilde{\Delta}(a,b,c)$ as the subgroup of index two of orientation-preserving transformations.  An embedding is given by $(x,y) \mapsto (\widetilde{x} \widetilde{y}, \widetilde{y} \widetilde{z}, \widetilde{z} \widetilde{x})$.

\subsection{Quaternion algebras}
In a series of papers in the 1970's~\cite{Takeuchi/69, Takeuchi/75, Takeuchi/77, Takeuchi/77JFS}, Takeuchi determined~\cite[Theorem~3]{Takeuchi/77} with computer assistance that there are precisely $85$ triples $(a,b,c)$ for which $\overline{\Delta}(a,b,c)$ is arithmetic.  Moreover, he provided an explicit construction of orders in quaternion algebras that realize this arithmeticity.  In this section, we give a brief overview of Takeuchi's work, stated in the explicit form needed for our computations.  As is standard, for a field $F$ we denote by $\left\langle \frac{x,y}{F} \right\rangle$ the $F$-algebra spanned by $\{1, i, j, ij \}$, with multiplication determined by the relations $i^2 = x, j^2 = y, ij = - ji$.  If $F$ is a totally real number field, then this algebra ramifies at an embedding $v : F \hookrightarrow \R$ if and only if $v(x) < 0$ and $v(y) < 0$.

Let $\Gamma \leq \mathrm{SL}_2(\R)$ be any discrete subgroup whose image in $\mathrm{PSL}_2(\R)$ has finite covolume.  Let $F_\Gamma$ be the field generated over $\Q$ by the traces of the elements of $\Gamma$. Under the hypotheses above, $F_\Gamma$ is a number field; by construction, it is equipped with a distinguished embedding $v_0: F_\tau \hookrightarrow \R$.  By Propositions~2 and~3 of~\cite{Takeuchi/69}, the $F_\Gamma$-vector subspace $F_\Gamma[\Gamma] \subseteq M_2(\R)$ is a quaternion algebra over $F_\Gamma$; if the traces of all elements of $\Gamma$ are algebraic integers, then $\mathcal{O}_{F_\Gamma} [\Gamma]$ is an order of $F_\Gamma[\Gamma]$.  

We assume henceforth that $\tau = (a,b,c)$ is a hyperbolic triple with $c < \infty$, since this hypothesis is needed for the explicit presentations of Propositions~\ref{pro:alpha.beta.fulldelta} and~\ref{pro:presentation.a} below.  See~\cite[\S5]{CV/19} for some treatment of the case $c = \infty$.  If we take $\Gamma = \iota(\Delta(\tau)) \leq \mathrm{SL}_2(\R)$, where $\iota$ is the embedding of~\eqref{equ:iota.embedding}, then
\begin{equation*}
F_\Gamma = F_{\iota(\Delta(\tau))} = \Q \left( \cos \frac{\pi}{a}, \cos \frac{\pi}{b}, \cos \frac{\pi}{c} \right).
\end{equation*}
We denote this field by $F_\tau$.  
It is clear from~\cite[Remark~5.24]{CV/19} that the quaternion algebra $B_\tau = F_\Gamma [\Gamma] = F_{\iota(\Delta(\tau))}[\iota(\Delta(\tau))]$ has the presentations
\begin{equation} \label{equ:b.presentation}
B_\tau = \left\langle \frac{4 \cos^2 \left( \frac{\pi}{s} \right) - 4, \delta}{F_\tau} \right\rangle,
\end{equation}
where $s$ is any element of $\{a,b,c \}$ and 
\begin{equation*}
\delta = 4 \left( \cos^2 \frac{\pi}{a} + \cos^2 \frac{\pi}{b} + \cos^2 \frac{\pi}{c} + 2 \cos \frac{\pi}{a} \cos \frac{\pi}{b} \cos \frac{\pi}{c} - 1 \right).
\end{equation*}
\begin{rem} \label{rmk:chebyshev}
Observe that $2 \cos \frac{\pi}{n} = \zeta_{2n} + \zeta_{2n}^{-1}$, where $\zeta_{2n}$ is a suitable primitive $2n$-th root of unity.  Hence $2 \cos \frac{\pi}{n}$ is an algebraic integer.  Consequently, $ \delta \in \mathcal{O}_{F_{\tau}}$; this is the reason for the factor of $4$ in the definition of $\delta$.  Note that $v_0(\delta) > 0$ whenever $\tau$ is hyperbolic.

Moreover, for any $m \in \N$ observe that $\cos \frac{m \pi}{n} = T_m \left( \cos \frac{\pi}{n} \right)$, where $T_m(x) \in \Z[x]$ is the Chebyshev polynomial of degree $m$.  Similarly, if $(m,n) = 1$, then $\cos \frac{\pi}{n}$ can be expressed as a polynomial in $\cos \frac{m \pi}{n}$.  It follows that $F_\tau = \Q \left( \cos \frac{m_a \pi}{a}, \cos \frac{m_b \pi}{b}, \cos \frac{m_c \pi}{c} \right)$ for any $m_a, m_b, m_c$ satisfying $(m_a, a) = (m_b, b) = (m_c, c) = 1$.  
\end{rem}
\begin{pro} \label{pro:alpha.beta.fulldelta}
Consider the elements $\alpha = 1 \cdot \iota(x)$ and $\beta = 1 \cdot \iota(y)$ of $B_\tau$, where $x,y \in \Delta(\tau)$ are as in~\eqref{equ:delta.presentation}.  Then $\{1, \alpha, \beta, \alpha \beta \}$ is an $\mathcal{O}_{F_\tau}$-basis of the order $\mathcal{O}_\tau = \mathcal{O}_{F_\tau} [\iota(\Delta(\tau))] \subset B_\tau$.  Moreover, there exist elements $i,j \in B_\tau$ providing
the presentation $B_\tau = \left\langle \frac{4 \cos^2 \left( \frac{\pi}{a} \right) - 4, \delta}{F_\tau} \right\rangle$, for which
\begin{eqnarray*}
\alpha & = & \cos \frac{\pi}{a} + \frac{1}{2} i \\
\beta & = & \cos \frac{\pi}{b} - \frac{\cos \frac{\pi}{a} \cos \frac{\pi}{b} + \cos \frac{\pi}{c}}{2 \left( \cos^2 \left( \frac{\pi}{a} \right) - 1 \right)} i + \frac{1}{4 \left( \cos^2 \left( \frac{\pi}{a} \right) - 1 \right)} ij.
\end{eqnarray*}
\end{pro}
\begin{proof}
The first claim is established in the discussion preceding~\cite[Lemma~5.4]{CV/19}.  The second follows easily from~\cite[Remark~5.24]{CV/19}, noting that the third element in the list given there of  elements of $B_\tau$
realizing the presentation above should read $(\lambda_{2a}^2 - 4)\delta_b + (\lambda_{2a} \lambda_{2b} + 2 \lambda_{2c})\delta_a - (\lambda_{2a}^2 \lambda_{2b} + \lambda_{2a} \lambda_{2c} - 2 \lambda_{2b})$ in the notation of~\cite{CV/19}.
\end{proof}

\begin{rem} \label{rmk:embedding}
By construction, the elements $\alpha, \beta \in B_\tau$ have reduced norm $1$, inducing embeddings $\Delta(\tau) \hookrightarrow (\mathcal{O}_\tau)^1$ and $\overline{\Delta}(\tau) \hookrightarrow (\mathcal{O}_{\tau})^1 / \{ \pm 1 \}$; here $\mathcal{O}_\tau^1$ is the group of elements of $\mathcal{O}_\tau$ of reduced norm $1$.  Since $v_0(\delta) > 0$, the quaternion algebra $B_\tau$ splits at the distinguished embedding $v_0$.  Thus induces an embedding $\Delta(\tau) \hookrightarrow (\mathcal{O}_\tau)^1 \subseteq (B_\tau)^1 \stackrel{v_0}{\hookrightarrow} \mathrm{SL}_2(\R)$, which we denote $\varepsilon_0$.
\end{rem}

Since the reduced traces of $\alpha$, $\beta$, and $\alpha \beta$ are $2 \cos \frac{\pi}{a}$, $2 \cos \frac{\pi}{b}$, and $-2 \cos \frac{\pi}{c}$, respectively, we verify that $\mathcal{O}_\tau$ is integral.  Indeed,
$$ \alpha \beta = - \cos \frac{\pi}{c} - \frac{\cos \frac{\pi}{b} + \cos \frac{\pi}{a} \cos \frac{\pi}{c}}{2 \left( \cos^2 \frac{\pi}{a} - 1 \right)} i + \frac{1}{2} j + \frac{\cos \frac{\pi}{a}}{4 \left( \cos^2 \frac{\pi}{a} - 1 \right)} ij.$$
We collect some facts about the order $\mathcal{O}_{\tau}$.

\begin{pro}[{\cite[Lemmas~5.4 and~5.5]{CV/19}}] \label{pro:discriminant}
The discriminant of $\mathcal{O}_{\tau}$ is the principal $\mathcal{O}_{F_\tau}$-ideal generated by $ \delta$.  Let $\mathfrak{P} \vartriangleleft \mathcal{O}_{F_\tau}$ be a prime ideal such that $\mathfrak{P} | \delta$.  Then $\mathfrak{P} | 2abc$.  Moreover, if the triple $\tau = (a,b,c)$ is not of the form $(mk, m(k+1), mk(k+1))$ for any $m,k \in \mathbb{N}$, then $\mathfrak{P} | abc$.
\end{pro}

Even if $\overline{\Delta}(\tau)$ is arithmetic, the quaternion algebra $B_\tau$ need not be split at exactly one infinite place of the totally real field $F_\tau$.  Define the subgroup $\Delta^{(2)}(\tau) \leq \Delta(\tau)$ generated by $-1$ and by the set $\{ \gamma^2 : \gamma \in \Delta(\tau) \}$.  This is a normal subgroup satisfying~\cite[(5.9)]{CV/19}
$$ \Delta(\tau) / \Delta^{(2)}(\tau) \simeq \begin{cases}
1 &: \, \text{at least two of} \, a,b,c \, \text{are odd,} \\
\Z / 2\Z &: \text{one of} \, a,b,c \, \text{is odd,} \\
\Z / 2\Z \times \Z / 2\Z &: a,b,c \, \text{are all even or $\infty$.}
\end{cases} 
$$

Now repeat the previous construction for the group $\Gamma = \iota(\Delta^{(2)}(\tau)) \leq \mathrm{SL}_2(\R)$.  By the union of~\cite[Proposition~4]{Takeuchi/75} and~\cite[Proposition~5]{Takeuchi/77}, the trace field is then 
\begin{equation} \label{equ:def.etau}
E_\tau := \Q \left( \cos^2 \frac{\pi}{a}, \cos^2 \frac{\pi}{b}, \cos^2 \frac{\pi}{c}, \cos \frac{\pi}{a} \cos \frac{\pi}{b} \cos \frac{\pi}{c} \right),
\end{equation}
which we view as a subfield of $\R$.
Let $A_\tau$ denote the quaternion algebra $E_\tau[\iota(\Delta^{(2)}(\tau))]$.  Write $\mathcal{Q}_\tau$ for the quaternion order $\mathcal{O}_{E_\tau}[\iota(\Delta^{(2)}(\tau))] \subset A_\tau$.  An explicit description of $A_\tau$ and $\mathcal{Q}_\tau$ is given by
the following analogue of Proposition~\ref{pro:alpha.beta.fulldelta}.
\begin{pro} \label{pro:presentation.a}
Consider the elements $\gamma_1 = 1 \cdot \iota(y^2)$ and $\gamma_2 = 1 \cdot \iota(z^2)$ of $A_\tau$, where $y,z \in \Delta(\tau)$ are as in~\eqref{equ:delta.presentation}.  Then $\{1, \gamma_1, \gamma_2, \gamma_1 \gamma_2 \}$ is an $\mathcal{O}_{E_\tau}$-basis of the order $\mathcal{Q}_\tau = \mathcal{O}_{E_\tau}[\iota(\Delta^{(2)}(\tau))] \subset A_\tau$.  Moreover, there exist elements $i,j \in A_\tau$ realizing the presentation 
\begin{equation*}
 A_\tau \simeq \left \langle \frac{16 \cos^2 \frac{\pi}{b} \left( \cos^2 \left( \frac{\pi}{b} \right) - 1 \right),
16 \delta \cos^2 \frac{\pi}{b} \cos^2 \frac{\pi}{c}}{E_\tau} \right \rangle
\end{equation*}
and satisfying
\begin{eqnarray*}
\gamma_1 & = & \cos \frac{2 \pi}{b} + \frac{i}{2} \\
\gamma_2 & = & \cos \frac{2 \pi}{c} + \frac{\cos \frac{2 \pi}{b} + \cos \frac{2 \pi}{b} \cos \frac{2 \pi}{c} + \cos \frac{2 \pi}{c} + 4 \cos \frac{\pi}{a} \cos \frac{\pi}{b} \cos \frac{\pi}{c} + 1}{2 \left( 1 - \cos^2 \frac{2 \pi}{b} \right)} i + \frac{1}{4 \left( 1 - \cos^2 \frac{2 \pi}{b} \right)} ij.
\end{eqnarray*}
\end{pro}
\begin{proof}
All this is stated explicitly in~\cite[Proposition~2]{Takeuchi/77JFS} (where $\gamma_1$ and $\gamma_2$ are denoted $\gamma_2^2$ and $\gamma_3^2$, respectively) except for the expressions for $\gamma_1$ and $\gamma_2$ as linear combinations of $1, i, j, ij$.  These are obtained by inverting the explicit transformation given in the proof of~\cite[Proposition~2]{Takeuchi/77JFS} and observing the following identities, where $c_1$, $c_2$, $c_3$ are as in~\cite{Takeuchi/77JFS}:
\begin{alignat*}{3}
c_1 & = & &\frac{1}{2} \mathrm{tr} \, \iota(y^2) & =  &\cos \frac{2 \pi}{b} \\
c_2 & = & &\frac{1}{2} \mathrm{tr} \, \iota(z^2) & =  &\cos \frac{2 \pi}{c} \\
c_3 & = & &\frac{1}{2} \mathrm{tr} \, \iota(y^2 z^2) & =  &1 - 2 \cos^2 \frac{\pi}{b} - 4 \cos \frac{\pi}{a} \cos \frac{\pi}{b} \cos \frac{\pi}{c} - 2 \cos^2 \frac{\pi}{c}. \qedhere
\end{alignat*} 
\end{proof}

It is immediate that $\mathrm{tr}(\mathcal{Q}_\tau) \subseteq \mathcal{O}_{E_\tau}$.  Therefore, $\overline{\Delta}(\tau)$ is \emph{semi-arithmetic}~\cite[Definition~3]{SchallerWolfart/00}.
Takeuchi shows~\cite[Theorem~1]{Takeuchi/77} that $\overline{\Delta}(\tau)$ is arithmetic if and only if the finite-index subgroup $\overline{\Delta}^{(2)}(\tau)$ is isomorphic to a finite-index subgroup of $G(A_\tau, \mathcal{Q}_\tau)$, if and only if $A_\tau$ ramifies at all infinite places of $E_\tau$ other than the distinguished one.  

\subsection{Congruence subgroups of triangle groups} \label{sec:congruence.subgroups}
As before, let $\tau = (a,b,c)$ be a hyperbolic triple with $c < \infty$.  Let $J \vartriangleleft \mathcal{O}_{E_\tau}$ be an ideal, and let $\mathfrak{J} \vartriangleleft \mathcal{O}_{F_\tau}$ be an ideal such that $\mathfrak{J} \cap \mathcal{O}_{E_\tau} = J$.  Suppose that $J$ is coprime to $abc$.  Moreover suppose that $\mathfrak{J}$ is coprime to the discriminant of $\mathcal{O}_\tau$; by Proposition~\ref{pro:discriminant}, this is only an additional assumption if $(a,b,c) = (mk, m(k+1), mk(k+1)$ for some $n,k \in \N$.  In this setup, Clark and Voight define a finite-index congruence subgroup $\overline{\Delta}(\tau;J) \trianglelefteq \overline{\Delta}(\tau)$ as the kernel of an explicit homomorphism $\Psi_J : \overline{\Delta}(\tau) \to \mathrm{PSL}_2(\mathcal{O}_{F_\tau} / \mathfrak{J})$ constructed in~\cite[\S5]{CV/19}.  If $\overline{\Delta}(\tau) \simeq G(B_\tau, \mathcal{O}_\tau)$, then $\Psi_J$ is surjective.  The significance of these groups is that if $J = \p$ is a prime of $E_\tau$ dividing a rational prime $p$, then the Riemann surface $X_{\overline{\Delta}(\tau;\p)}$ admits a Belyi map $X_{\overline{\Delta}(\tau;\p)} \to \mathbb{P}^1(\C)$, namely a non-constant morphism ramified at exactly three points, which is also a Galois cover.

If $k$ is a finite field and $k^\prime / k$ is the extension of degree two, then the homomorphism $\mathrm{GL}_2(k) \to \mathrm{PSL}_2(k^\prime)$ given by $g \mapsto \pm g / \sqrt{\det g}$ induces an embedding $\mathrm{PGL}_2(k) \hookrightarrow \mathrm{PSL}_2(k^\prime)$; note that every element of $k$ is a square in $k^\prime$.
The finite quotient $\overline{\Delta}(\tau) / \overline{\Delta}(\tau;\p)$ is described explicitly in~\cite[Theorem~9.1]{CV/19}:

\begin{pro}[Clark-Voight] \label{pro:CV.congruence.quotient}
Suppose that $\tau = (a,b,c)$ and the prime $p$ are as above.  Let $\mathfrak{p}$ be a prime  of $E_\tau$ dividing $p$, and let $k_\p = \mathcal{O}_{E_\tau} / \p$ be its residue field.  Then
$$ \overline{\Delta}(\tau) / \overline{\Delta}(\tau; \p) \simeq
\begin{cases}
\mathrm{PSL}_2 (k_\p) &: \p \text{ splits completely in the extension } F_\tau / E_\tau \\
\mathrm{PGL}_2 (k_\p) &: \text{otherwise.}
\end{cases}
$$
\end{pro}

It is easy to see that $F_\tau$ is contained in the cyclotomic field $\Q(\zeta_{2abc})$, where $\zeta_{2abc}$ is a primitive $2abc$-th root of unity.  Hence $F_\tau / \Q$ is an abelian Galois extension.  Therefore $E_\tau / \Q$ is also Galois, and the quotient $\overline{\Delta}(\tau) / \overline{\Delta}(\tau;\p)$ is independent, up to isomorphism, of the choice of prime ideal $\p \vartriangleleft \mathcal{O}_{E_\tau}$ dividing a rational prime $p$.  Set $K(\tau;p) = \overline{\Delta}(\tau) / \overline{\Delta}(\tau;\p)$.

The area of a fundamental domain of $\overline{\Delta}(\tau;J)$ is $2\pi \left( 1 - \frac{1}{a} - \frac{1}{b} - \frac{1}{c} \right) [\overline{\Delta}(\tau) : \overline{\Delta}(\tau;J)]$.  By the Gauss-Bonnet theorem, the genus of the Riemann surface $X_{\overline{\Delta}(\tau;J)}$ is thus
\begin{equation} \label{equ:genus}
g(X_{\overline{\Delta}(\tau;J)}) = \frac{[\overline{\Delta}(\tau) : \overline{\Delta}(\tau;J)]}{2} \left( 1 - \frac{1}{a} - \frac{1}{b} - \frac{1}{c} \right) + 1.
\end{equation}

\begin{rem} \label{rmk:pgl.psl}
We will view $K(\tau;p) = \Psi_\p(\overline{\Delta}(\tau))$ as a subgroup of $\mathrm{PSL}_2(\mathcal{O}_{F_\tau} / \mathfrak{P})$, where $\mathfrak{P}$ is a prime of $F_\tau$ dividing $\p$.  From now on we will omit $\tau$ from the notation and write $\overline{\Delta}$ and $\overline{\Delta}(\p)$ when this should not cause confusion.  
\end{rem}

\subsection{Normal subgroups of $\overline{\Delta}(\tau)$} \label{sec:normal.subgroups}
Let $\tau = (a,b,c)$ as before.  We will be interested below in classifying the normal subgroups $H \trianglelefteq \overline{\Delta}(\tau)$ satisfying 
\begin{equation} \label{equ:right.quotient}
\overline{\Delta}(\tau)/H \simeq K(\tau;p).
\end{equation}
Let $T(\tau; p)$ be the set of triples $(\gamma_1, \gamma_2, \gamma_3) \in K(\tau; p)^3$ such that $\gamma_1^a = \gamma_2^b = \gamma_3^c = \gamma_1 \gamma_2 \gamma_3 = e$ and such that $\langle \gamma_1, \gamma_2, \gamma_3 \rangle = K(\tau;p)$.  The automorphism group $\mathrm{Aut}(K(\tau;p))$ acts on $T(\tau;p)$ in the obvious way.  Each element $(\gamma_1,\gamma_2,\gamma_3) \in T(\tau;p)$ gives rise to a normal subgroup $H \trianglelefteq \overline{\Delta}(\tau)$ satisfying~\eqref{equ:right.quotient}, namely the kernel of the epimorphism $\psi : \overline{\Delta}(\tau) \twoheadrightarrow K(\tau;p)$ determined by $\psi(x) = \gamma_1, \psi(y) = \gamma_2$, where $x,y \in \overline{\Delta}(\tau)$ are generators as in~\eqref{equ:triangle.group}.  Clearly two elements in the same orbit of the $\mathrm{Aut}(K(\tau;p))$-action produce the same normal subgroup, and every normal subgroup satisfying~\eqref{equ:right.quotient} arises in this way.  Thus the number of normal subgroups satisfying~\eqref{equ:right.quotient} is bounded by $| T(\tau;p) / \mathrm{Aut}(K(\tau;p))|$.
This quantity can be studied by means of the theory of Macbeath~\cite{Macbeath/69}, as extended by Clark and Voight~\cite{CV/19}.  We present here just enough to state the results we will use and refer the reader to~\cite{Macbeath/69} and~\cite[\S6-8]{CV/19} for details.

For any field $k$, let $\pi : \mathrm{SL}_2(k) \to \mathrm{PSL}_2(k)$ denote the natural projection.  For any $n \in \mathbb{N}$ and any prime power $q$, define $\kappa(n,q)$ to be the number of conjugacy classes of $\mathrm{PSL}_2(\F_q)$ whose elements have order $n$.  

\begin{lem} \label{lem:eigenvalues}
Let $\F_q$ be a finite field of odd characteristic $p$, and let $n \in \mathbb{N}$ be coprime to $p$.  Then $\kappa(n,q) = 0$ if $q^2 \not\equiv 1 \, \mathrm{mod} \, 2n$.  Otherwise,
$$ \kappa(n,q) = \begin{cases}
1 &: n \in \{ 1, 2 \} \\
\frac{\varphi(n)}{2} &: n \geq 3 \, \mathrm{odd} \\
\frac{\varphi(2n)}{4} &: n \geq 4 \, \mathrm{even},
\end{cases}
$$
where $\varphi$ is Euler's totient function.  The common trace of the elements in each conjugacy class of $\mathrm{PSL}_2(\F_q)$ with elements of order $n$ has the form $\pm (\zeta + \zeta^{-1})$, where $\zeta$ is a primitive $n$-th root of unity if $n$ is odd, and a primitive $2n$-th root of unity if $n$ is even.
\end{lem}
\begin{proof}
Observe first that if $q^2 \not\equiv 1 \, \mathrm{mod} \, 2n$, then, under our hypotheses, $n \nmid | \mathrm{PSL}_2(\F_q)|$ and hence $\kappa(n,q) = 0$.  Now let $\gamma \in \mathrm{PSL}_2(\F_q)$ have order $n$, and let $M \in \pi^{-1}(\gamma) \subset \mathrm{SL}_2(\F_q)$.  Then $M$ is semisimple, and its eigenvalues are reciprocal primitive $2n$-th roots of unity if $n$ is even and reciprocal primitive $n$-th or $2n$-th roots of unity if $n$ is odd.  It is a simple exercise to count the possible traces, up to sign.  Since a semisimple element of $\mathrm{PSL}_2(\F_q)$ is determined up to conjugacy by its trace, this completes the proof.
\end{proof}

Consider a triple $\underline{t} = (t_1, t_2, t_3) \in \F_q^3$.  We will refer to such objects as trace triples.  Let $T(\underline{t})$ be the set of triples $(g_1, g_2, g_3) \in \mathrm{SL}_2(\F_q)$ such that $g_1 g_2 g_3 = e$ and $\tr g_i = t_i$ for all $i \in \{1,2,3 \}$; note that we do not make any hypothesis here about the orders of the $g_i$ or the subgroup of $\mathrm{SL}_2(\F_q)$ generated by them.  A fundamental result of Macbeath~\cite[Theorem~1]{Macbeath/69} is that $T(\underline{t}) \neq \varnothing$ for all trace triples $\underline{t}$.  Moreover, Macbeath classified trace triples as follows.  We say that $\underline{t}$ is commutative if there exists $(g_1, g_2, g_3) \in T(\underline{t})$ such that $\pi(\langle g_1, g_2, g_3 \rangle)$ is an abelian subgroup of $\mathrm{PSL}_2(\F_q)$.  Commutative trace triples are easily identified in practice thanks to the following consequence of~\cite[Corollary~1]{Macbeath/69}.

\begin{lem}[Macbeath] \label{lem:commutative}
Let $\underline{t} = (t_1, t_2, t_3) \in \F_q^3$ be a trace triple.  Then $\underline{t}$ is commutative if and only if $t_1^2 + t_2^2 + t_3^2 - t_1 t_2 t_3 - 4 = 0$.
\end{lem}

Let $\underline{g} = (g_1, g_2, g_3) \in \mathrm{SL}_2(\F_q)^3$.  We associate to it the order triple $o(\underline{g}) \in \N^3$ consisting of the orders of the elements $\pi(g_1), \pi(g_2), \pi(g_3) \in \mathrm{PSL}_2(\F_q)$, arranged in non-descending order.  A trace triple $\underline{t}$ is called exceptional if there exists $\underline{g} \in T(\underline{t})$ such that
\begin{multline*}
o(\underline{g}) \in \{ (2,2,c) : c \geq 2 \} \cup \\ \{ (2,3,3),(2,3,4),(2,3,5),(2,5,5),(3,3,3),(3,3,5),(3,4,4),(3,5,5),(5,5,5) \}.
\end{multline*}
If $\underline{t}$ is not commutative, then $o(\underline{g})$ is constant on all $\underline{g} \in T(\underline{t})$.  We say that $\underline{t}$ is projective if, for all $\underline{g} \in T(\underline{t})$, the subgroup $\pi(\langle g_1, g_2, g_3 \rangle) \leq \mathrm{PSL}_2(\F_q)$ is isomorphic to $\mathrm{PSL}_2(k)$ or $\mathrm{PGL}_2(k)$ for some subfield $k \subseteq \F_q$.  Macbeath~\cite[Theorem~4]{Macbeath/69} proved that every trace triple $\underline{t}$ is commutative, exceptional, or projective.  

Recall that the trace of an element of $\mathrm{PSL}_2(\F_q)$ is only defined up to sign.  Reflecting this and following~\cite[\S8]{CV/19}, we define a trace triple up to sign to be a triple $\pm \underline{t} = (\pm t_1, \pm t_2, \pm t_3)$, where $t_1, t_2, t_3 \in \F_q$.  A triple $(t_1^\prime, t_2^\prime, t_3^\prime) \in \F_q^3$ is called a lift of $\pm \underline{t}$ if $t_i^\prime = \pm t_i$ for all $i \in \{ 1, 2, 3 \}$, where the signs in each component may be taken independently.  We say that a trace triple up to sign is commutative or exceptional if it has a lift with the corresponding property.  Following the terminology of~\cite[\S8]{CV/19}, we say that $\pm \underline{t}$ is partly projective if it has a lift which is projective.  Every trace triple up to signs is commutative, exceptional, or partly projective~\cite[Lemma~8.9]{CV/19}.

Let $\underline{C} = (C_1, C_2, C_3)$ be a triple of conjugacy classes of $\mathrm{PSL}_2(\F_q)$.  Again following~\cite{CV/19}, define $\Sigma(\underline{C})$ to be the set of triples $\underline{\gamma} = (\gamma_1, \gamma_2, \gamma_3) \in \mathrm{PSL}_2(\F_q)^3$ such that $\gamma_1 \gamma_2 \gamma_3 = e$, that $\gamma_i \in C_i$ for all $i \in \{ 1,2,3 \}$, and that $\langle \gamma_1, \gamma_2, \gamma_3 \rangle = K(\tau;\p)$. 
The natural action of $\mathrm{Aut}(K(\tau;p))$ on triples of elements need not preserve $\Sigma(\underline{C})$.  However, we may still define an equivalence relation on $\Sigma(\underline{C})$ by $\underline{\gamma} \sim \underline{\gamma}^\prime$ if there exists $\sigma \in \mathrm{Aut}(K(\tau; p))$ such that $\sigma(\gamma_i) = \gamma_i^\prime$ for all $i$.  Denote the set of equivalence classes by $\Sigma(\underline{C})/\mathrm{Aut}(K(\tau;p))$.  Let $o(\underline{C})$ be the triple of the common orders of the elements of $C_1$, $C_2$, and $C_3$, arranged in non-descending order, and let $\tr \underline{C}$ be the trace triple up to signs associated to $\underline{C}$.  Let $\F_p(\tr \underline{C})$ be the subextension of $\F_q / \F_p$ generated by the components of $\tr \underline{C}$.

We are now ready to state a bound for the number of normal subgroups $H \trianglelefteq \overline{\Delta}(\tau)$ such that $\overline{\Delta}(\tau) / H \simeq K(\tau;p)$; see~\cite[Theorem~1.6]{Sah/69} for an overlapping result.  Suppose that $K(\tau;p) \leq \mathrm{PSL}_2(\F_q)$ as in Remark~\ref{rmk:pgl.psl}.  Let $\mathcal{C}(\tau;p)$ be the set of triples $\underline{C} = (C_1,C_2,C_3)$ of conjugacy classes of $\mathrm{PSL}_2(\F_q)$ such that $\Sigma(\underline{C}) \cap T(\tau;p) \neq \varnothing$.

\begin{pro} \label{pro:counting.normal.subgroups}
Let $\tau = (a,b,c)$, and let $p > 2$ be a rational prime.  
Suppose that for every $\underline{C} \in \mathcal{C}(\tau;p)$ the associated trace triple up to signs $\tr \underline{C}$ is partly projective and not exceptional, and that it satisfies $\F_q = \F_p(\tr \underline{C})$.  Set $\Omega(\tau;p) = \{ o(\underline{C}) : \underline{C} \in \mathcal{C}(\tau;p) \}$ and $\mathcal{N}(\tau;p) =  \left\{ H \trianglelefteq \overline{\Delta}(\tau) : \overline{\Delta}(\tau) / H \simeq K(\tau;p) \right\}$.  Then
$$
 | \mathcal{N}(\tau;p) | \leq \sum_{(a^\prime, b^\prime, c^\prime) \in \Omega(\tau;p) \atop a^\prime = 2} \kappa(a^\prime,q) \, \kappa(b^\prime,q) \, \kappa(c^\prime,q) + \sum_{(a^\prime, b^\prime, c^\prime) \in \Omega(\tau;p) \atop a^\prime > 2} 2 \, \kappa(a^\prime,q) \, \kappa(b^\prime,q) \, \kappa(c^\prime,q).
$$
\end{pro}
\begin{proof}
It is clear that
$$  | \mathcal{N}(\tau;p) | \leq |T(\tau;p)/\mathrm{Aut}(K(\tau;p))| \leq \sum_{\underline{C} \in \mathcal{C}(\tau;p)} |\Sigma(\underline{C})/\mathrm{Aut}(K(\tau;p))|,$$
where the first inequality follows from the discussion at the beginning of this section.  The second inequality need not be an equality since the action of $\mathrm{Aut}(K(\tau;p))$ on $T(\tau;p)$ need not preserve the sets $\Sigma(\underline{C}) \cap T(\tau;p)$.  Our hypotheses ensure that~\cite[Proposition~8.10]{CV/19} applies; from it we conclude that $| \Sigma(\underline{C})/\mathrm{Aut}(K(\tau;p))| = 1$ if $o(\underline{C})$ has a component equal to $2$ and $| \Sigma(\underline{C})/\mathrm{Aut}(K(\tau;p))| \leq 2$ otherwise.  This completes the proof.
\end{proof}

\begin{rem}
Since $a^\prime b^\prime c^\prime | abc$, it is clear that if $p \nmid abc$, then the integers $\kappa(a^\prime,q)$, $\kappa(b^\prime,q)$, and $\kappa(c^\prime,q)$ appearing on the right-hand side of the conclusion of Proposition~\ref{pro:counting.normal.subgroups} may be computed from the formula in Lemma~\ref{lem:eigenvalues}.  We also note that the prime $p = 2$ has been excluded from the discussion in this section only to streamline the exposition; analogous results may be obtained in this case by analogous arguments.  
\end{rem}

\section{Semi-arithmeticity} \label{sec:semiarithmetic}
As mentioned above, only finitely many of the hyperbolic triangle groups $\overline{\Delta}(\tau)$ are arithmetic.  However, $\overline{\Delta}(\tau)$ is semi-arithmetic for all hyperbolic triples $\tau$ in the sense of~\cite[Definition~3]{SchallerWolfart/00}.  In this section we apply a recent argument of Cosac and D\'{o}ria~\cite[Theorem~1.5]{CosacDoria/20} to obtain, in some cases, a lower bound on $\mathrm{sys}(X_{\overline{\Delta}(\tau;J)})$ in terms of the genus $g(X_{\overline{\Delta}(\tau;J)})$, which is given explicitly by~\eqref{equ:genus}, and an invariant of an explicit quaternion algebra.  We then discuss earlier work in the arithmetic case, where more precise bounds are sometimes available.

Recall the quaternion algebra $A_\tau$ of Proposition~\ref{pro:presentation.a}.  It is defined over the totally real field $E_\tau$ of~\eqref{equ:def.etau}, which is the invariant trace field of $\overline{\Delta}(\tau)$.  Let $r_\tau$ be the number of infinite places of $E_\tau$ at which $A_\tau$ splits.  Takeuchi~\cite[\S3]{Takeuchi/75} proved that $\overline{\Delta}(\tau)$ is arithmetic if and only if $r_\tau = 1$ and that this condition holds for exactly $85$ triples $\tau$.  Nugent and Voight~\cite{NV/17} later showed that there are finitely many hyperbolic triples $\tau$ satisfying $r_\tau = N$ for any $N \in \N$ and provided an algorithm for computing them.

Cohen and Wolfart~\cite[\S2]{CohenWolfart/90} observed that the triangle group $\overline{\Delta}(\tau)$ admits a modular embedding.  This means (cf.~\cite[Definition~4]{SchallerWolfart/00}) that there is an arithmetic lattice $\Lambda < \mathrm{PSL}_2(\R)^{r_\tau}$, defined over $E_\tau$, a holomorphic map $F : \mathcal{H} \to \mathcal{H}^{r_\tau}$, and embeddings $\Phi_1, \dots, \Phi_{r_\tau} : \overline{\Delta}(\tau) \hookrightarrow \mathrm{PSL}_2(\R)$, extending each of the embeddings $E_\tau \hookrightarrow \R$ at which $A_\tau$ splits, such that the map $\Phi = (\Phi_1, \dots, \Phi_{r_\tau}) : \overline{\Delta}(\tau) \to \mathrm{PSL}_2(\R)^{r_\tau}$ satisfies $\Phi(\overline{\Delta}(\tau)) \subseteq \Lambda$ and $F(\gamma(z)) = \Phi(\gamma)(F(z))$ for all $\gamma \in \overline{\Delta}(\tau)$ and $z \in \mathcal{H}$.  The lattice $\Lambda$ is commensurate with a subgroup of an arithmetic lattice arising from a quaternion algebra, but it need not in general be contained in such a lattice.
More explicitly, Clark and Voight~\cite[Proposition~5.13]{CV/19} construct an embedding $\overline{\Delta}(\tau) \hookrightarrow N_{A_\tau}(\mathcal{Q}_\tau)/E_\tau^\times$, where $\mathcal{Q}_\tau \subset A_\tau$ is the order defined in Proposition~\ref{pro:presentation.a} and $N_{A_\tau}(\mathcal{Q}_\tau)$ is its normalizer in $A_\tau$; see also~\cite[Remark~3(ii)]{SchallerWolfart/00}.

However, if at least two of the components of the triple $\tau$ are odd, then the fields $E_\tau = F_\tau$ coincide, as do the quaternion algebras $A_\tau = B_\tau$ and the orders $\mathcal{Q}_\tau = \mathcal{O}_\tau$; see~\eqref{equ:b.presentation} and Proposition~\ref{pro:alpha.beta.fulldelta} for the definitions.  In this case there is an embedding $\overline{\Delta}(\tau) \hookrightarrow \mathcal{Q}^1_\tau / \{ \pm 1 \}$ as in Remark~\ref{rmk:embedding}, and the congruence subgroups are $\overline{\Delta}(\tau;J) = \overline{\Delta}(\tau) \cap \mathcal{Q}_\tau^1 (J)$ for all ideals $J \vartriangleleft \mathcal{O}_{E_\tau}$ coprime to $2abc$, where $\mathcal{Q}_\tau^1 (J) = \{ x \in \mathcal{Q}_\tau^1 : x \equiv 1 \, \mathrm{mod} \, J \mathcal{Q}_\tau \}$.  Note that $\mathcal{Q}_\tau^1 (J)$ embeds in $\mathcal{Q}_\tau^1 / \{ \pm 1 \}$ by our assumptions on $J$.  Then an argument of Cosac and D\'{o}ria obtains the following systolic bound.

\begin{thm} \label{thm:semiarithmetic}
Let $\tau = (a,b,c)$ be a hyperbolic triple with $c < \infty$.  Suppose that at least two of $a,b,c$ are odd.  There exists a constant $c_\tau$ such that, for every ideal $J \vartriangleleft \mathcal{O}_{E_\tau}$ coprime to $2abc$, the following holds:
\begin{equation} \label{equ:semiarithmetic}
 \mathrm{sys}(X_{\overline{\Delta}(\tau;J)}) \geq \frac{4}{3r_\tau} \log g(X_{\overline{\Delta}(\tau;J)}) - c_\tau.
 \end{equation}
\end{thm}
\begin{proof}
We follow the proof of~\cite[Theorem~1.5]{CosacDoria/20}.
Let $J = \prod_{i = 1}^r \p_i^{e_i}$, where $\p_1, \dots, \p_r$ are distinct prime ideals of $\mathcal{O}_{E_\tau}$ that are coprime to $2abc$.  We write $N(J)$ for the norm of $J$.  By the proof of~\cite[Proposition~9.7]{CV/19}, we have
\begin{equation} \label{equ:general.quotient} 
\overline{\Delta}(\tau) / \overline{\Delta}(\tau;J) \simeq \prod_{i = 1}^r \mathrm{PSL}_2(\mathcal{O}_{E_\tau} / \p_i^{e_i}).
\end{equation}  
Setting $q_i = N(\p_i) = | \mathcal{O}_{E_\tau} / \p_i |$, we obtain
\begin{equation} \label{equ:norm.bound}
 [\overline{\Delta}(\tau) : \overline{\Delta}(\tau;J)]  = \prod_{i = 1}^r |\mathrm{PSL}_2(\mathcal{O}_{E_\tau} / \p_i^{e_i})|  = \prod_{i = 1}^r \frac{(q_i^3 - q_i) q_i^{3(e_i - 1)}}{2}  <  \prod_{i = 1}^r q_i^{3e_i} = N(J)^3.
 \end{equation}
Arguing as in the proof of~\cite[Theorem~1.5]{CosacDoria/20} and observing that we may take $\mathcal{F}$ to be the family of all ideals $J \vartriangleleft \mathcal{O}_{E_\tau}$ coprime to $2abc$, we conclude that if $N({J}) \geq 2^{[E_\tau:\Q]}$ then 
$$ \mathrm{sys}(X_{\overline{\Delta}(\tau;J)}) \geq \frac{4}{{r_\tau}} \log \mathrm{vol}(X_{\mathcal{Q}_\tau^1({J})}) - c_\tau^\prime = \frac{4}{{r_\tau}} \log N({J}) - c_\tau^{\prime \prime},$$
for constants $c_\tau^\prime, c_\tau^{\prime \prime}$ independent of $J$.  Hence it follows from~\eqref{equ:genus} and~\eqref{equ:norm.bound} that 
$$ \mathrm{sys}(X_{\overline{\Delta}(\tau;J)}) \geq \frac{4}{3 r_\tau} \log g(X_{\overline{\Delta}(\tau;J)}) - c_\tau $$
for all $J \vartriangleleft \mathcal{O}_{E_\tau}$ coprime to $2abc$ such that $N(J) \geq 2^{[E_\tau: \Q]}$ and a constant $c_\tau$ independent of $J$.  Since there are only finitely many ideals $J$ with $N(J) < 2^{[E_\tau: \Q]}$, we obtain the claim after possibly increasing $c_\tau$.
\end{proof}

\begin{rem}
If $\tau$ has at most one odd component, then the image of the embedding $j: \overline{\Delta}(\tau) \hookrightarrow N_{A_\tau}(\mathcal{Q}_\tau)/E_\tau^\times$ is not in general contained in $\mathcal{Q}_\tau^1 / \{ \pm 1 \}$.  In this case a straightforward modification of the proof of Theorem~\ref{thm:semiarithmetic} establishes~\eqref{equ:semiarithmetic} for all ideals $J \vartriangleleft \mathcal{O}_{E_\tau}$ such that $J$ is coprime to $2abc$ and $j(\overline{\Delta}(\tau;J)) \subseteq \mathcal{Q}_\tau^1 / \{ \pm 1 \}$.  It would be interesting to resolve whether Theorem~\ref{thm:semiarithmetic} is valid for all hyperbolic triples.
\end{rem}

\begin{rem}
In the case of arithmetic Fuchsian groups $\Gamma$, the result of Cosac and D\'{o}ria~\cite[Theorem~1.5]{CosacDoria/20} cited above recovers an earlier theorem of Katz, Schaps, and Vishne~\cite[Theorem~1.5]{KSV/07}, which gave the bound $\mathrm{sys}(X_{\Gamma(J)}) \geq \frac{4}{3} g(X_{\Gamma(J)}) - c_{\Gamma}$.  Here $\Gamma$ is commensurate with $\mathcal{O}^1 / \{ \pm 1 \}$ for an order $\mathcal{O}$ in a quaternion algebra defined over a totally real field $F$ and split at exactly one infinite place, and $\Gamma(J) = \Gamma \cap \mathcal{O}^1(J)$.
This generalized earlier work of Buser and Sarnak~\cite[(4.7)]{BS/94} in the case $F = \Q$.  Makisumi~\cite[Theorem~1.6]{Makisumi/13} has shown that the coefficient $\frac{4}{3}$ is the best possible in a result of this form. 
See, for instance,~\cite{MaclachlanReid/03} for details on arithmetic Fuchsian groups.  Indeed, Theorem~\ref{thm:semiarithmetic} in the case $r_\tau = 1$ was known to us before~\cite{CosacDoria/20} appeared.
\end{rem}

In general Theorem~\ref{thm:semiarithmetic} does not provide an effective lower bound on $\mathrm{sys}(X_{\overline{\Delta}(\tau;J)})$, since no effective constraint on the constant $c_\tau$ is known.  In some arithmetic cases, it is known that one may take $c_\tau = 0$; see~\cite[Theorem~1.10]{KSV/07} and~\cite[Proposition~9.1]{KKSV/16} for sufficient conditions.  These results rely on the trace bounds of~\cite[Theorem~2.3]{KSV/07} and do not readily extend to non-arithmetic situations.  The upper bounds on $\mathrm{sys}(X_{\overline{\Delta}(\tau;J)})$ obtained by the method of Section~\ref{sec:procedure} below show that one may {\emph{not}} take $c_\tau = 0$ for some specific triples $\tau$.  Finally, we note that the bounds given by Theorem~\ref{thm:semiarithmetic} are, in general, transcendental.  However, it follows from Remark~\ref{rmk:embedding} that traces of elements of $\overline{\Delta}(\tau;J)$, and hence $\mathrm{sys}(X_{\overline{\Delta}(\tau;J)})$ and our upper bound, are contained in the totally real field $F_\tau$ and are thus algebraic numbers.  Hence one cannot expect to determine the exact values of the systolic lengths of triangular modular curves by matching the upper bounds of Section~\ref{sec:procedure} with lower bounds arising from Theorem~\ref{thm:semiarithmetic}.

\section{Computations} \label{sec:computations}
\subsection{Procedure} \label{sec:procedure}
Let $\tau = (a,b,c)$ be a hyperbolic triple, let $\overline{\Delta} = \overline{\Delta}(\tau)$, and let $\p \vartriangleleft \mathcal{O}_{E_\tau}$ be a prime ideal such that the congruence subgroup $\overline{\Delta}(\p)$ is defined.  Our aim is to bound the systolic length of the compact Riemann surface $X_{\overline{\Delta}(\p)}$.  In fact, for any finite real number $M > 0$, the set of numbers at most $M$ that occur as lengths of closed geodesics of $X_{\overline{\Delta}(\p)}$ could likely be computed by adapting the method used by Vogeler~\cite{Vogeler/03} in the Hurwitz case $\tau = (2,3,7)$; see also earlier work~\cite{DT/00, LW/02, Woods/01} by undergraduate participants of an REU at Rose-Hulman directed by S.~A.~Broughton.

We take a simpler but less precise approach that was introduced in~\cite{KKSV/16} in the case of $\tau = (3,3,4)$.  Recall from Section~\ref{sec:traces} that determining $\mathrm{sys}(X_{\overline{\Delta}(\tau;\p)})$ amounts to finding
$$ \min \left\{ \left\{ | \mathrm{tr} \, \varepsilon(\tilde{\gamma}) | : \gamma \in \overline{\Delta}(\tau;\p) \right\} \cap (2, \infty) \right\},$$
where $\gamma$ runs over all elements of the infinite group $\overline{\Delta}(\tau;\p)$, where $\tilde{\gamma} \in \Delta(\tau)$ is a lift of $\gamma$, and $\varepsilon: \Delta(\tau) \hookrightarrow \mathrm{SL}_2(\R)$ is some embedding; by work of Takeuchi mentioned above, the result is independent of the choice of $\varepsilon$.  

We restrict to a set of Schreier generators of $\overline{\Delta}(\tau;\p)$.  This is a {\emph{finite}} generating set with some pleasant properties; details may be found in classical references on combinatorial group theory such as~\cite[Theorem~2.9]{MKS/66}.  Among the generating sets of a subgroup of a finitely presented group obtainable by available algorithms, the elements of a set of Schreier generators are expressible as relatively short words in the generators $\alpha$ and $\beta$ of $\overline{\Delta}(\tau)$.  Thus we may hope that the hyperbolic Schreier generators will have relatively low traces.  The minimal trace of a hyperbolic Schreier generator is our upper bound on $\mathrm{sys}(X_{\overline{\Delta}(\tau;\p)})$.  Of course, it is possible that some other hyperbolic element of $\overline{\Delta}(\tau)$ has a lower trace than that of any of our Schreier generators.  However, as mentioned in the introduction, in all the examples that we have computed where $\mathrm{sys}(X_{\overline{\Delta}(\tau;\p)})$ is known, it matches our bound.

The procedure behind our computations is presented below.  We assume that $\p \vartriangleleft \mathcal{O}_{E_\tau}$ is prime as this streamlines the discussion of possible simplifications of the various steps of the computation.  However, the same procedure may be used to find upper bounds on $\mathrm{sys}(X_{\overline{\Delta}(\tau;J)})$ for composite ideals $J$; in this case the description of the quotient $\overline{\Delta}(\tau)/ \overline{\Delta}(\tau;J)$ given in~\eqref{equ:general.quotient} should be used in Step 1 instead of Proposition~\ref{pro:CV.congruence.quotient}.

Our code is available at {\texttt{github.com/mmschein/systole}}.

{\bf{Step 1:}} {\emph{Identify subgroups $H \trianglelefteq \overline{\Delta}(\tau)$ such that $\overline{\Delta}(\tau)/H \simeq K(\tau;p)$.}}

If $|K(\tau;p)|$ is reasonably small, then Magma's {\texttt{LowIndexNormalSubgroups}} routine may be used.  When $K(\tau;p) = \mathrm{PSL}_2(\F_q)$ with $q > 3$, it follows from the classification of finite simple groups that $K(\tau;p)$ is the unique simple group of order $\frac{q}{(2,q-1)} (q^2 - 1)$ and so it is particularly easy to identify relevant subgroups $H \trianglelefteq \overline{\Delta}(\tau)$ from the output of {\texttt{LowIndexNormalSubgroups}}.   

If $|K(\tau;p)|$ is large, it is more efficient to generate random pairs of elements $(z_1, z_2) \in K(\tau;p)^2$ and select those for which $(z_1, z_2, (z_1 z_2)^{-1}) \in T(\tau;p)$; recall that $T(\tau;p)$ was defined in Section~\ref{sec:normal.subgroups}.  Such a pair gives rise to an epimorphism $\psi: \overline{\Delta}(\tau) \twoheadrightarrow K(\tau;p)$ determined by $\psi(x) = z_1$ and $\psi(y) = z_2$.  Then we consider the kernel $H = \ker \psi$.

{\bf{Step 2:}} {\emph{Compute a set of Schreier generators for $H$.}}

See~\cite[\S4]{Neubuser/82} and~\cite[\S6]{Sims/94} for algorithms computing Schreier generators.  We have used the implementation in Magma.  This is by far the most time-consuming step of our computations.

{\bf{Step 3:}} {\emph{Determine whether there is an ideal $\p \vartriangleleft \mathcal{O}_{E_\tau}$ dividing $p$ such that $H = \overline{\Delta}(\tau;\p)$, and identify $\p$ if so.}}

The most general method for doing this is as follows.  Recall from Remark~\ref{rmk:embedding} that we have an explicit embedding $\varepsilon_0: {\Delta}(\tau) \hookrightarrow \mathcal{O}_\tau^1$.  Let $S_H$ be the set of Schreier generators of $H$ computed in the previous step.  For every $s \in S_H$, we compute $\varepsilon_0(s) = u_0 + u_1 i + u_2 j + u_3 ij$, in terms of the presentation of $B_\tau$ given in Proposition~\ref{pro:alpha.beta.fulldelta}.  Set
$$
\left( \begin{array}{c} v_0 \\ v_1 \\ v_2 \\ v_3 \end{array} \right) = 
\left( \begin{array}{cccc} 
1 & \cos \frac{\pi}{a} & \cos \frac{\pi}{b} & - \cos \frac{\pi}{c} \\
0 & \frac{1}{2} & - \frac{\cos \frac{\pi}{a} \cos \frac{\pi}{b} + \cos \frac{\pi}{c}}{2 \left( \cos^2 \frac{\pi}{a} - 1 \right)} & - \frac{\cos \frac{\pi}{b} + \cos \frac{\pi}{a} \cos \frac{\pi}{c}}{2 \left( \cos^2 \frac{\pi}{a} - 1 \right)} \\
0 & 0 & 0 & \frac{1}{2} \\
0 & 0 & \frac{1}{4 \left( \cos^2 \frac{\pi}{a} - 1 \right)} & \frac{\cos \frac{\pi}{a}}{4 \left( \cos^2 \frac{\pi}{a} - 1 \right)}
\end{array} \right)^{-1}
\left( \begin{array}{c} u_0 \\ u_1 \\ u_2 \\ u_3 \end{array} \right).
$$
Then the $v_i$ are the coefficients of $\varepsilon_0(s)$ with respect to the $F_\tau$-basis $(1, \alpha, \beta, \alpha \beta)$ of $B_\tau$.  If, for all $s \in S_H$, we have $v_0 \equiv \pm 1 \, \mathrm{mod} \, \mathfrak{P}$ and $v_1, v_2, v_3 \in \mathfrak{P}$ for a prime ideal $\mathfrak{P} \vartriangleleft \mathcal{O}_{F_\tau}$ dividing $\p$, then $H \leq \overline{\Delta}(\tau;\p)$ and hence $H = \overline{\Delta}(\tau;\p)$ as they are subgroups of $\overline{\Delta}(\tau)$ of the same index.

This step of the computation can often be simplified.  For instance, the upper bound on $\mathcal{N}(\tau;p) =  \left\{ H \trianglelefteq \overline{\Delta}(\tau) : \overline{\Delta}(\tau) / H \simeq K(\tau;p) \right\}$ obtained from Proposition~\ref{pro:counting.normal.subgroups} is sometimes equal to the number of prime ideals of $\mathcal{O}_{E_\tau}$ dividing $p$.  In this case, any subgroup $H$ obtained in Step 1 is necessarily a congruence subgroup, and testing just a few Schreier generators will suffice to eliminate all $\p | p$ but one.

If, futhermore, $\mathcal{O}_{F_\tau} / \mathfrak{P} = \mathcal{O}_{E_\tau} / \p$ (or, equivalently, $K(\tau;p) = \mathrm{PSL}_2(k_\p)$), then we may dispense with Schreier generators entirely in this step.  Indeed, the explicit homomorphism $\Psi_\p: \overline{\Delta}(\tau) \to \mathrm{PSL}_2(\mathcal{O}_{F_\tau} / \mathfrak{P})$, whose kernel is $\overline{\Delta}(\tau;\p)$, is given in~\cite[\S5]{CV/19}.  Here $\mathfrak{P}$ is a prime of $F_\tau$ dividing $\p$.
We know by~\cite[Proposition~5.23]{CV/19} that 
 $\Psi_\p(x)$, $\Psi_\p(y)$, and $\Psi_\p((xy)^{-1})$ have traces $\pm (2 \cos \frac{\pi}{a}) \, \mathrm{mod} \, \mathfrak{P}$, $\pm (2 \cos \frac{\pi}{b}) \, \mathrm{mod} \, \mathfrak{P}$, and $\pm (2 \cos \frac{\pi}{c}) \, \mathrm{mod} \, \mathfrak{P}$, respectively.  Then in Step 1 we search for an epimorphism $\psi: \overline{\Delta}(\tau) \twoheadrightarrow \mathrm{PSL}_2(k_\p)$ satisfying these trace conditions and can conclude immediately that $\ker \psi = \overline{\Delta}(\tau;\p)$.  It is still necessary to determine the Schreier generators of $\overline{\Delta}(\tau;\p)$ for the following step.
 
 {\bf{Step 4:}} {\emph{Determine the minimal trace of a Schreier generator.}}
 
 For every Schreier generator $s \in S_H$, we compute $| \mathrm{tr} \, \varepsilon(\tilde{s}) |$, where $\tilde{s} \in \Delta(\tau)$ is a lifting of $s$ and $\epsilon: \Delta(\tau) \hookrightarrow \mathrm{SL}_2(\R)$ is some embedding.  It is usually convenient to take $\varepsilon$ to be the explicit embedding of Remark~\ref{rmk:embedding}.  Our upper bound on $\mathrm{sys}(X_{\overline{\Delta}(\tau;\p)})$ is then 
 $$ \min \left\{ \left\{ | \mathrm{tr} \, \varepsilon(\tilde{s}) | : s \in S_{\overline{\Delta}(\tau;\p)} \right\} \cap (2, \infty) \right\}.$$

\begin{rem}
Let $g$ be the genus of $X_{\overline{\Delta}(\tau;\p)}$; recall that an explicit formula for $g$ was given in~\eqref{equ:genus}.  In all cases where we have computed it, Reidemeister-Schreier rewriting as implemented in Magma indeed produces the simplest presentation of $\overline{\Delta}(\tau;\p)$, with $2g$ generators and a single relation of length $4g$ in which each generator and its inverse appear exactly once.
\end{rem}

In the remainder of this section we consider some specific examples of the computations outlined above.

\subsection{Hurwitz surfaces} \label{sec:hurwitz}
A Hurwitz surface is a compact Riemann surface $X$ of genus $g(X)$ with exactly $84(g(X) - 1)$ automorphisms, which is the maximal number possible by a classical theorem of Hurwitz.  These have been studied since 1879, when Klein discovered his celebrated quadric of genus three, with $168$ automorphisms.  Together with the theorem mentioned above, Hurwitz showed~\cite[\S7]{Hurwitz/93} that a group of order $84(g-1)$ is the automorphism group of some Riemann surface of genus $g$ if and only if it is a homomorphic image of the triangle group $\overline{\Delta}(2,3,7)$.  Thus the new Hurwitz surfaces that have been discovered over the past century have amounted to constructions of normal subgroups of $\overline{\Delta}(2,3,7)$ of finite index; as a very incomplete sample of these results, we mention the work of Sinkov~\cite{Sinkov/37}, Macbeath~\cite{Macbeath/61, Macbeath/65}, Leech~\cite{Leech/65}, Lehner and Newman~\cite{LN/67}, and Cohen~\cite{Cohen/79}.  The construction of congruence subgroups of $\overline{\Delta}(\tau)$ discussed above, in the special case $\tau = (2,3,7)$, recovers some, but not all, of this previous work.

Set $\mu = 2 \cos \frac{\pi}{7}$ for brevity.  If $\tau = (2,3,7)$, then $F_\tau = E_\tau = \Q(\mu)$.  This is a cubic field, where $\mu$ has minimal polynomial $x^3 - x^2 - 2x + 1$ over $\Q$.  Let $\zeta_7$ be a primitive seventh root of unity such that $- (\zeta_7 + \zeta_7^{-1}) = \mu$.  Then $\Q(\mu)$ is contained in the cyclotomic field $\Q(\zeta_7)$.  In particular, $\mathcal{O}_{F_\tau} = \Z [\mu]$.  We determine the decomposition of primes in $\Q(\mu)$.

\begin{lemma} \label{lem:hurwitz.decomposition}
The prime $7$ is totally ramified in $\Q(\mu)$.  If $p \neq 7$ is a rational prime, then $p$ splits completely in $\Q(\mu)$ if $p \equiv \pm 1 \, \mathrm{mod} \, 7$, and $p$ is inert otherwise.
\end{lemma}
\begin{proof}
We rely on basic facts about cyclotomic fields; see, for instance,~\cite[{\S}I.10]{Neukirch/92}.  The only prime ramifying in $\Q(\zeta_7)$ is $7$, and it is totally ramified.  Therefore the same is true for the subfield $\Q(\mu)$.  Now assume $p \neq 7$, let $\p$ be a prime of $\Q(\mu)$ dividing $p$, and let $\mathfrak{P}$ be a prime of $\Q(\zeta_7)$ dividing $\p$.  Since $\Q(\mu)/\Q$ is a cubic Galois subextension of the sextic Galois extension $\Q(\zeta_7) / \Q$, it is clear that $p$ is completely split in $\Q(\mu)$ if and only if the inertia degree is $f(\p / p) = 1$, which in turn is equivalent to $f(\mathfrak{P} / p) \in \{1, 2 \}$.  Otherwise, $f(\p / p) = 3$ and $p$ is inert in $\Q(\mu)$.  Now $f(\mathfrak{P} / p)$ is the minimal $f \in \N$ such that $p^f \equiv 1 \, \mathrm{mod} \, 7$.  Thus $f(\mathfrak{P}/p) \in \{1, 2 \}$ if and only if $7 | (p^2 - 1)$, which is equivalent to $p \equiv \pm 1 \, \mathrm{mod} \, 7$.
\end{proof}

Since $\overline{\Delta}(2,3,7) = \overline{\Delta}^{(2)} (2,3,7)$, we have $\overline{\Delta}(2,3,7) \leq \mathcal{O}_\tau^1 / \{ \pm 1 \}$ as a subgroup of finite index, where by Proposition~\ref{pro:alpha.beta.fulldelta} the order $\mathcal{O}_\tau \subset B_\tau$ of the quaternion algebra $B_\tau = \left \langle \frac{-4, \mu^2 - 3}{F_\tau} \right \rangle$ has $\mathcal{O}_{F_\tau}$-basis $\{ 1, \alpha, \beta, \alpha \beta \}$, with $\alpha = \frac{i}{2}$ and $\beta = \frac{1}{2} + \frac{\mu}{4} i - \frac{1}{4}ij$.  In fact, $\overline{\Delta}(2,3,7)$ is arithmetic~\cite[Theorem~3]{Takeuchi/77}, so $B_\tau$ splits only at the distinguished infinite place of $F_\tau$.

\begin{rem}
The construction given here as a special case of Takeuchi's general theory is equivalent to the ``Hurwitz order'' $\mathcal{Q}_{\mathrm{Hur}}$ studied by Katz, Schaps, and Vishne~\cite[(2.8)]{KSV/11}.  Indeed, they denote $\eta = 2 \cos \frac{2 \pi}{7} = \mu^2 - 2$.  Observe that $\Q(\mu) = \Q(\eta)$ by Remark~\ref{rmk:chebyshev} and that
\begin{eqnarray*}
\varphi : \left \langle \frac{-4, \mu^2 - 3}{\Q(\mu)} \right \rangle & \to & \left \langle \frac{\eta, \eta}{\Q(\eta)} \right \rangle \\
\varphi(i) & = & (2 \mu - 2) ij \\
\varphi(j) & = & -(\mu^2 - \mu - 1) i 
\end{eqnarray*}
is an isomorphism of quaternion algebras between our $B_\tau$ and the algebra denoted $D$ in~\cite{KSV/11}.  Moreover, $\varphi(\alpha) = g_2$ and $\varphi(\beta) = g_3$, where $g_2, g_3 \in D$ are as defined in~\cite[\S4]{KSV/11}; they arise from work of Elkies~\cite{Elkies/99}.  Since $\mathcal{Q}_{\mathrm{Hur}}$ is spanned by $\{1, g_2, g_3, g_2 g_3 \}$ over $\Z[\mu]$ by~\cite[Theorem~4.2]{KSV/11}, we have $\varphi(\mathcal{O}_\tau) = \mathcal{Q}_{\mathrm{Hur}}$.
\end{rem}

We have $\mathcal{O}_\tau^1 / \{ \pm 1 \} = \simeq \overline{\Delta} = \overline{\Delta}(2,3,7)$~\cite[\S4.4]{Elkies/99}.  The order $\mathcal{O}_\tau \subset B_\tau$ is maximal and splits at all finite places of $\Q(\mu)$ since the quaternion algebra $B_\tau$ does so.  Hence for any ideal $I \vartriangleleft \mathcal{O}_{F_{\tau}} = \Z [\mu]$ we get a congruence subgroup $\overline{\Delta}(I) \vartriangleleft \overline{\Delta}$ such that $\overline{\Delta} / \overline{\Delta}(I) \simeq \mathrm{PSL}_2 (\mathcal{O}_{F_\tau} / I)$.  By~\cite[Proposition~9.1]{KKSV/16} we conclude that $\mathrm{sys} ( X_{\overline{\Delta}(I)}) > \frac{4}{3} g(X_{\overline{\Delta}(I)})$ for all but finitely many ideals $I \vartriangleleft \Z [\mu]$; indeed, $T_1 = T_2 = \varnothing$ in the notation of~\cite{KKSV/16} and  $2^6 = 64 < 84 = \frac{4 \pi}{2\pi \left( 1 - \frac{1}{2} - \frac{1}{3} - \frac{1}{7} \right)}$. By~\cite[Remark~9.2]{KKSV/16} this inequality holds for all $I$ such that $g(X_{\overline{\Delta}(I)}) \geq 85$.  

\begin{pro} \label{pro:hurwitz.normal.subgroups}
Let $p$ be prime, and let $H \trianglelefteq \overline{\Delta} = \overline{\Delta}(2,3,7)$ be a normal subgroup.  Then
$$ \overline{\Delta} / H \simeq \begin{cases}
\mathrm{PSL}_2(\F_p) &: p \equiv 0, \pm 1 \, \mathrm{mod} \, 7 \\
\mathrm{PSL}_2(\F_{p^3}) &: p \not\equiv 0, \pm 1 \, \mathrm{mod} \, 7
\end{cases}
$$
if and only if $H = \overline{\Delta}(\p)$, where $\p \vartriangleleft \Z[\mu]$ is a prime dividing $p$.
\end{pro}
\begin{proof}
One direction is clear by the previous paragraph and Lemma~\ref{lem:hurwitz.decomposition}.  
The other direction is contained in~\cite[Theorem~8]{Macbeath/69} and historically was the first application of Macbeath's classification.  We give the proof in some detail, as an illustration of the machinery of Section~\ref{sec:normal.subgroups}, although it is essentially the same as the one in~\cite{Macbeath/69}.

We first compute, using Magma's \texttt{LowIndexNormalSubgroups} function, that there is a unique normal subgroup $H \trianglelefteq \overline{\Delta}$ for which the quotient $\overline{\Delta}/H$ is isomorphic to each of $\mathrm{PSL}_2(\F_7)$, $\mathrm{PSL}_2(\F_8)$, and $\mathrm{PSL}_2(\F_{27})$.  These facts can also be proven without reliance on a computer by a suitable modification of the arguments below, but we take the easier path.  Thus we may henceforth assume $p \not\in \{2,3,7 \}$.  It suffices to show that there are at most three normal subgroups $H \trianglelefteq \overline{\Delta}$ satisfying $\overline{\Delta} / H \simeq \mathrm{PSL}_2(\F_p)$ if $p$ splits in $\Q(\mu)$, and at most one normal subgroup satisfying $\overline{\Delta} / H \simeq \mathrm{PSL}_2(\F_{p^3})$ if $p$ is inert.  

Let $\p$ be a place of $\Q(\mu)$ dividing $p$.  Then $K(\tau;p) = \mathrm{PSL}_2(\F_q)$, where $q = p$ if $p$ splits in $\Q(\mu)$ and $q = p^3$ otherwise.  Let $(\gamma_1, \gamma_2, \gamma_3) \in T(\tau;p)$, in the notation of Section~\ref{sec:normal.subgroups}.  We claim that $\gamma_1$, $\gamma_2$, and $\gamma_3$ have order $2$, $3$, and $7$, respectively, as elements of $\mathrm{PSL}_2(\F_q)$.  Indeed, otherwise one of the $\gamma_i$ would be trivial, and consequently the subgroup $\langle \gamma_1, \gamma_2, \gamma_3 \rangle$ would be trivial, contradicting the definition of $T(\tau;p)$.  Thus $\Omega(\tau;p) = \{ (2,3,7) \}$, in the notation of Proposition~\ref{pro:counting.normal.subgroups}.  Moreover, it follows from Lemma~\ref{lem:eigenvalues} that $\tr \gamma_1 = 0$ and $\tr \gamma_2 = \pm 1$, whereas $\tr \gamma_3$ is one of $\pm(\zeta_7 + \zeta_7^{-1})$, $\pm(\zeta_7^2 + \zeta_7^{-2})$, or $\pm(\zeta_7^3 + \zeta_7^{-3})$, where $\zeta_7$ is a primitive seventh root of unity.  Thus there are three potential trace triples up to sign associated to elements of $T(\tau;p)$.  By Lemma~\ref{lem:commutative}, none of them are commutative.  Moreover, they are not exceptional, since for any lift $\underline{t}$ of one of them and any $\underline{g} \in T(\underline{t})$, we have $o(\underline{g}) = (2,3,7)$.  Hence all three trace triples up to signs are partly projective.  It now follows from Proposition~\ref{pro:counting.normal.subgroups} and Lemma~\ref{lem:eigenvalues} that 
\begin{equation} \label{equ:hurwitz.count}
 | \mathcal{N}(\tau;p) | \leq \kappa(2,q) \, \kappa(3,q) \, \kappa(7,q) = \frac{\varphi(3) \, \varphi(7)}{4} = 3.
\end{equation}
Hence there are at most three subgroups $H \trianglelefteq \overline{\Delta}(\tau)$ satisfying $\overline{\Delta}(\tau) / H \simeq \mathrm{PSL}_2(\F_q)$.

It remains to show that there is only one such normal subgroup if $p$ is inert in $\Q(\mu)$.  In this case, let $\sigma(\alpha) = \alpha^p$ be the Frobenius automorphism of $\F_q = \F_{p^3}$.  The action of $\sigma$ on matrix elements induces an automorphism of $\mathrm{PSL}_2(\F_q)$.  Observe that 
$$\{ \sigma(\pm(\zeta_7 + \zeta_7^{-1})), \sigma^2(\pm(\zeta_7 + \zeta_7^{-1})) \} = \{ \pm(\zeta_7^2 + \zeta_7^{-2}), \pm (\zeta_7^3 + \zeta_7^{-3}) \}.$$  
Thus, if the epimorphism $\psi : \overline{\Delta}(\tau) \twoheadrightarrow \mathrm{PSL}_2(\F_q)$ corresponds to an element of $T(\tau;p)$ having one of the three possible trace triples up to signs, the epimorphisms $\sigma \circ \psi$ and $\sigma^2 \circ \psi$, which have the same kernel, correspond to the other two trace triples up to signs.  Equivalently, a single orbit of the action of $\langle \sigma \rangle \leq \mathrm{Aut}(K(\tau;p))$ intersects $\Sigma(\underline{C}) \cap T(\tau;p)$ for all three elements $\underline{C} \in \mathcal{C}(\tau;p)$.  In either case, we see from the proof of Proposition~\ref{pro:counting.normal.subgroups} that $\mathcal{N}(\tau;p)| = 1$.
\end{proof}

We tabulate below our bounds on $\mathrm{sys}(X_{\overline{\Delta}(I)})$ for all proper ideals $I \vartriangleleft \mathcal{O}_{F_\tau}$ of norm $N(I) \leq 100$, to an accuracy of three decimal places.  For all the prime ideals in this table, Vogeler~\cite[Appendix~C]{Vogeler/03} has computed length spectra of the corresponding Riemann surfaces.  In all cases, the lowest trace identified by us corresponds to the smallest geodesic length found by Vogeler.  His tables are complete for lengths smaller than $14.49$, so they determine the systolic length for surfaces $X$ with $\mathrm{sys}(X) < 14.49$.  This is evidence that our method provides a reasonable estimate for the systolic length $\mathrm{sys}(X_\Gamma)$.  

We see as in~\eqref{equ:genus} that $g(X_{\overline{\Delta}(I)}) = \frac{| \mathrm{PSL}_2(\Z[\mu]/I) |}{84} + 1$ for all $I \vartriangleleft \Z[\mu]$.  
The Riemann surface $X_{\overline{\Delta}(I)}$ of genus $3$ corresponding to the ideal $I$ of norm $7$ in the first line of the table is, of course, none other than the Klein quadric.  The surface $X_{\overline{\Delta}((2))}$ of genus $7$ was studied by Fricke~\cite{Fricke/99} and more than sixty years later by Macbeath~\cite{Macbeath/65}, who was apparently unaware of Fricke's work; see Serre's letter~\cite{Serre/90} to Abhyankar for a more modern perspective on this curve.  If $N$ is a natural number with a unique prime factor $p$ such that $p$ splits or ramifies in $\Q(\mu)$ (i.e. $p \equiv 0, \pm 1 \, \mathrm{mod} \, 7$) and if $(p^\prime)^3 | N$ for all prime factors $p^\prime \neq p$ of $N$, then there exist three ideals $I \vartriangleleft \Z[\mu]$ of norm $N(I) = N$, giving rise to three Hurwitz surfaces of the same genus; these are the Hurwitz triplets.

Not all normal subgroups of $\overline{\Delta}$ arise as $\overline{\Delta}(I)$ for some ideal $I \vartriangleleft \Z[\mu]$.  The example of lowest index is a pair of normal subgroups of index $1344$ found by Sinkov~\cite{Sinkov/37}.  They are conjugate in the full triangle group $\widetilde{\Delta}(2,3,7)$ of~\eqref{equ:full.triangle.group} and correspond to a single chiral Riemann surface of genus $\frac{1344}{84} + 1 = 17$.  In this case, the lowest trace we obtain is $11 \mu^2 + 9 \mu - 7$, corresponding to the systolic length found by Vogeler.  Cohen~\cite[Theorem~2]{Cohen/79} described an infinite family of normal subgroups of $\overline{\Delta}$ whose quotients are not isomorphic to $\mathrm{PSL}_2(\Z[\mu] / I)$ for any ideal $I \vartriangleleft \Z[\mu]$.
\begin{center}
\begin{longtable}{| r | r || r | r || r | r |} \hline
$I$ & $N(I)$ & $x =$ lowest observed trace & $2 \arcosh \frac{x}{2}$ & $g(X_{\overline{\Delta}(I)})$ & $\frac{4}{3} \log g$ \\ \hline \hline
$(\mu + 2)$ & $7$ & $ 2 \mu^2 + \mu - 1$ & $3.936$ & $3$ & $1.465$ \\ \hline
$(2)$ & $8$ & $4 \mu^2 + 4 \mu - 2$ & $5.796$ & $7$ & $2.595$ \\ \hline
$(2 \mu + 1)$ & $13$ & $4 \mu^2 + 4 \mu - 1$ & $5.904$ & $14$ & $3.519$ \\ \hline
$(2 \mu + 3)$ & $13$ & $6 \mu^2 + 5 \mu - 4$ & $6.393$ & $14$ & $3.519$ \\ \hline
$(\mu - 3)$ & $13$ & $7 \mu^2 + 7 \mu - 4$ & $6.888$ & $14$ & $3.519$ \\ \hline
$(3)$ & $27$ & $45 \mu^2 + 36 \mu - 25$ & $10.451$ & $118$ & $6.361$ \\ \hline
$(4 \mu - 1)$ & $29$ & $19 \mu^2 + 15 \mu - 12$ & $8.680$ & $146$ & $6.645$ \\ \hline
$(3 \mu - 4)$ & $29$ & $49 \mu^2 + 41 \mu - 27$ & $10.656$ & $146$ & $6.645$ \\ \hline
$(\mu + 3)$ & $29$ & $73 \mu^2 + 60 \mu - 40$ & $11.442$ & $146$ & $6.645$ \\ \hline
$(4 \mu - 3)$ & $41$ & $33\mu^2 + 26 \mu - 17$ & $9.340$ & $411$ & $8.025$ \\ \hline
$(3 \mu + 1)$ & $41$ & $49 \mu^2 + 40 \mu - 26$ & $10.648$ & $411$ & $8.025$ \\ \hline
$(\mu - 4)$ & $41$ & $121 \mu^2 + 97 \mu - 67$ & $12.432$ & $411$ & $8.025$ \\ \hline
$(2 \mu - 5)$ & $43$ & $49 \mu^2 + 40 \mu - 28$ & $10.628$ & $474$ & $8.215$ \\ \hline
$(3 \mu + 2)$ & $43$ & $55 \mu^2 + 43 \mu - 32$ & $10.824$ & $474$ & $8.215$ \\ \hline
$(5 \mu - 3)$ & $43$ & $86 \mu^2 + 69 \mu - 48$ & $11.747$ & $474$ & $8.215$ \\ \hline
$(\mu + 2)^2$ & $49$ & $105 \mu^2 + 84 \mu - 58$ & $12.147$ & $687$ & $8.710$ \\ \hline
$(2)(\mu + 2)$ & $56$ & $184 \mu^2 + 148 \mu - 102$ & $13.272$ & $1009$ & $9.222$ \\ \hline
$(2)^2$ & $64$ & $80 \mu^2 + 64 \mu - 46$ & $11.593$ & $1537$ & $9.783$ \\ \hline
$(5 \mu - 1)$ & $71$ & $121 \mu^2 + 97 \mu - 66$ & $12.436$ & $2131$ & $10.219$ \\ \hline
$(\mu + 4)$ & $71$ & $133 \mu^2 + 106 \mu - 73$ & $12.619$ & $2131$ & $10.219$ \\ \hline
$(4 \mu - 5)$ & $71$ & $151 \mu^2 + 121 \mu - 83$ & $12.877$ & $2131$ & $10.219$ \\ \hline
$(3 \mu + 5)$ & $83$ & $276 \mu^2 + 220 \mu - 153$ & $14.077$ & $3404$ & $10.844$ \\ \hline
$(8 \mu - 3)$ & $83$ & $384 \mu^2 + 308 \mu - 213$ & $14.742$ & $3404$ & $10.844$ \\ \hline
$(5 \mu - 8)$ & $83$ & $520 \mu^2 + 416 \mu - 289$ & $15.346$ & $3404$ & $10.844$ \\ \hline
$(\mu + 2)(\mu - 3)$ & $91$ & $44 \mu^2 + 36 \mu - 25$ & $10.416$ & $2185$ & $10.252$ \\ \hline
$(\mu + 2)(2 \mu + 1)$ & $91$ & $88 \mu^2 + 72 \mu - 49$ & $11.808$ & $2185$ & $10.252$ \\ \hline
$(\mu + 2)(2 \mu + 3)$ & $91$ & $256 \mu^2 + 205 \mu - 143$ & $13.928$ & $2185$ & $10.252$ \\ \hline
$(3\mu - 7)$ & $97$ & $321 \mu^2 + 257 \mu - 179$ & $14.380$ & $5433$ & $11.467$ \\ \hline
$(4 \mu + 3)$ & $97$ & $412 \mu^2 + 331 \mu - 228$ & $14.884$ & $5433$ & $11.467$ \\ \hline
$(7 \mu - 4)$ & $97$ & $711 \mu^2 + 569 \mu - 395$ & $15.972$ & $5433$ & $11.467$ \\ \hline
\end{longtable}

\end{center}

\subsection{$(2,3,12)$ triangle surfaces} \label{sec:2312}
We next consider the case $\tau = (a,b,c) = (2,3,12)$.  A Riemann surface $X$ of genus $g$ is called maximal if it is a local maximum for the function $\mathrm{sys}(X)$ on the Teichm\"{u}ller space $T_g$.  The $(2,3,12)$ triangle surfaces of genus $3$ and $4$ were shown to be maximal by Schmutz Schaller~\cite[\S8]{Schmutz/93}, so this is a natural family in which to search for surfaces with large systolic length.    

In this section, write $\overline{\Delta}$ for $\overline{\Delta}(2,3,12)$.  Observe that $E_\tau = \Q(\sqrt{3})$ and $F_\tau = \Q(\sqrt{2}, \sqrt{3})$.  The complication of this case, in comparison with the Hurwitz case considered in the previous section, is that now $\overline{\Delta}$ is not contained in $\mathcal{O}^1 / \{ \pm 1 \}$ for any quaternion order $\mathcal{O}$ defined over the invariant trace field $E_\tau$; only the subgroup $\overline{\Delta}^{(2)}$ is.  

\begin{lem} \label{lem:quotient2312}
Let $\overline{\Delta} = \overline{\Delta}(2,3,12)$, and let 
$p \geq 5$ be a rational prime.  Then
$$
K(\tau;p) \simeq
\begin{cases}
\mathrm{PSL}_2(\F_p) &: p \equiv 1,23 \, \mathrm{mod} \, 24 \\
\mathrm{PSL}_2(\F_{p^2}) &: p \equiv 5,7,17,19 \, \mathrm{mod} \, 24 \\
\mathrm{PGL}_2(\F_p) &: p \equiv 11, 13 \, \mathrm{mod} \, 24.
\end{cases}
$$
\end{lem}
\begin{proof}
The claim follows from Proposition~\ref{pro:CV.congruence.quotient} and basic algebraic number theory.
\end{proof}

\begin{pro} \label{pro:2312.normal.subgroups}
Let $p \geq 5$ be a rational prime and suppose that $H \trianglelefteq \overline{\Delta} = \overline{\Delta}(2,3,12)$ is a normal subgroup such that $\overline{\Delta} / H \simeq K(\tau;p)$.
Then there exists a prime ideal $\p$ of $\Z [ \sqrt{3} ]$ dividing $p$ such that $H = \overline{\Delta}(\p)$.
\end{pro}
\begin{proof}
We will show that there exist at most two normal subgroups $H \vartriangleleft \overline{\Delta}$ satisfying $\overline{\Delta} / H \simeq K(\tau;p)$ if $p$ splits in $\Q(\sqrt{3})$ and at most one such subgroup $H$ if $p$ is inert in $\Q(\sqrt{3})$.  Since the congruence subgroups $\overline{\Delta}(\p)$ satisfy $\overline{\Delta} / \overline{\Delta}(\p) \simeq K(\tau;p)$ by definition, it follows that no other subgroups satisfy this condition.

Following the strategy of the proof of Proposition~\ref{pro:hurwitz.normal.subgroups}, consider $(\gamma_1, \gamma_2, \gamma_3) \in T(\tau;p)$.  
Clearly the orders of $\gamma_1, \gamma_2, \gamma_3$ divide $2$, $3$, and $12$, respectively.  If one of these elements were trivial, 
the three of them would generate a cyclic group, contradicting the assumption that they generate $K(\tau;p)$.  Hence $o(\gamma_1) = 2$ and $o(\gamma_2) = 3$.  If $o(\gamma_3) < 12$, then the non-solvable group $K(\tau;p)$ would be a homomorphic image of $\overline{\Delta}(2,2,3)$, $\overline{\Delta}(2,3,4)$, or $\overline{\Delta}(2,3,6)$, which is absurd since these three groups are solvable; see~\cite[Example~2.4]{CV/19} and the paragraph following it.  Hence $o(\gamma_3) = 12$.  We have shown that $\Omega(\tau;p) = \{ (2,3,12) \}$.

By Lemma~\ref{lem:eigenvalues}, there are two conjugacy class triples $\underline{C}$ satisfying $o(\underline{C}) = (2,3,12)$; the corresponding trace triples up to signs are $(0, \pm 1, \pm(\zeta_{24} + \zeta_{24}^{-1}))$ and $(0, \pm 1, \pm(\zeta_{24}^5 + \zeta_{24}^{-5}))$, where $\zeta_{24}$ is a primitive $24$-th root of unity.  These are not commutative, by Lemma~\ref{lem:commutative}, and are not exceptional.  It now follows from Proposition~\ref{pro:counting.normal.subgroups} that $| \mathcal{N}(\tau;p) | \leq 2$.

It remains to show that if $p$ is inert in $\Q(\sqrt{3})$, i.e.~if $p \equiv \pm 5 \, \mathrm{mod} \, 12$, then $| \mathcal{N}(\tau;p)| = 1$.  In this case, the Frobenius automorphism $\sigma \in \mathrm{Gal}(\F_{p^2} / \F_p)$ interchanges $\pm (\zeta_{24} + \zeta_{24}^{-1})$ and $\pm (\zeta_{24}^5 + \zeta_{24}^{-5})$, and we proceed by the same argument as for the analogous case of Proposition~\ref{pro:hurwitz.normal.subgroups}.
\end{proof}

Observe that $\Delta(2,3,12)^{(2)} \simeq \Delta(3,3,6)$.  
We read off from Proposition~\ref{pro:presentation.a} that $A_\tau \simeq \left \langle \frac{-3, 1 + \sqrt{3}}{\Q(\sqrt{3})} \right \rangle$ and that a $\Z[\sqrt{3}]$-basis of the order $\mathcal{Q}_\tau$ is given by $\{ 1, \gamma_1, \gamma_2, \gamma_1 \gamma_2 \}$, where
\begin{eqnarray*}
\gamma_1 & = & \frac{-1 + i}{2} \\
\gamma_2 & = & \frac{\sqrt{3}}{2} + \frac{2 + \sqrt{3}}{6} i + \frac{ij}{3}.
\end{eqnarray*}
Moreover, $\mathcal{Q}_\tau^1 / \{ \pm 1 \} \simeq \overline{\Delta}(3,3,6)$; see the table on p.~208 of~\cite{Takeuchi/77JFS}.  However, $\overline{\Delta}(2,3,12)$ is isomorphic to the larger group $\mathcal{Q}_\tau^+ / (E_\tau^\times \cap \mathcal{Q}_\tau^+)$, where $\mathcal{Q}_\tau^+ \leq \mathcal{Q}_\tau^\times$ is the subgroup of elements with totally positive reduced norm.  Then $\mathcal{Q}_\tau^+ / (E_\tau^\times \cap \mathcal{Q}_\tau^+)$ naturally contains $\mathcal{Q}_\tau^1 / \{ \pm 1 \}$ as a subgroup of index two.

We tabulate the results of our computations for all prime ideals $\p \vartriangleleft \Z[\sqrt{3}]$ such that $N(\p) < 60$.  Note that if $p \in \{ 2, 3 \}$, then $p$ ramifies in $\Q(\sqrt{3})$ and also satisfies $p | abc$, so that Proposition~\ref{pro:2312.normal.subgroups} is inapplicable to primes dividing $p$.  However, it is still possible to define congruence subgroups; see~\cite[Remark~5.24]{CV/19}.  If $\p_2$ is the prime ideal dividing $2$, one finds explicitly that $\{ g \in \mathcal{Q}_\tau^+ : g \equiv 1 \, \mathrm{mod} \, \p_2 \mathcal{Q}_\tau \} \trianglelefteq \mathcal{Q}_\tau^+$ corresponds to a normal subgroup $H_2 \trianglelefteq \overline{\Delta}$ such that the quotient $\overline{\Delta}/H_2$ is a group of order $48$ isomorphic to the unique non-trivial extension of $A_4$ by $\Z / 4\Z$.  This $H$ is isomorphic to a surface group of genus $3$.  Similarly, for the prime ideal dividing $3$ one obtains a normal subgroup $H_3 \trianglelefteq \overline{\Delta}$ such that $\overline{\Delta} / H_3 \simeq S_4 \times \Z / 3\Z$ and $H_3$ is isomorphic to a surface group of genus $4$.  
The resulting Riemann surfaces of genera $3$ and $4$ are the maximal surfaces mentioned at the beginning of this section.  Their systolic lengths were computed by Schmutz Schaller~\cite{Schmutz/93} and match our estimates.

One sees by explicit computation that $\mathcal{Q}_\tau$ is a maximal order in $A_\tau$, and that the only finite place at which $A_\tau$ ramifies is $\p_2$.  It follows that both sides of the inequality in the hypotheses of~\cite[Proposition~9.1]{KKSV/16} are equal to $12$, and we are unable to conclude that $\mathrm{sys}(X_{\overline{\Delta}(\p)}) > \frac{4}{3} g(X_{\overline{\Delta}(\p)})$ for any $\p$; this depends on the coefficient of the second-order term of the series appearing in the proof of~\cite[Proposition~9.1]{KKSV/16}.  Nevertheless, from our computational data this appears plausible.

\begin{center}
\begin{longtable}{| r | r || r | r || r | r |} \hline
$\p$ & $N(\p)$ & $x =$ lowest observed trace & $2 \arcosh \frac{x}{2}$ & $g(X_{\overline{\Delta}(\p)})$ & $\frac{4}{3} \log g$ \\ \hline \hline
$(1 + \sqrt{3})$ & $2$ & $ 4 + 2\sqrt{3}$ & $3.983$ & $3$ & $1.465$ \\ \hline
$(\sqrt{3})$ & $3$ & $5 + 3 \sqrt{3}$ & $4.624$ & $4$ & $1.848$ \\ \hline
$(1 - 2 \sqrt{3})$ & $11$ & $31 + 19 \sqrt{3}$ & $8.314$ & $56$ & $5.367$ \\ \hline
$(1 + 2 \sqrt{3})$ & $11$ & $36 + 21 \sqrt{3}$ & $8.563$ & $56$ & $5.367$ \\ \hline
$(4 - \sqrt{3})$ & $13$ & $35 + 20 \sqrt{3}$ & $8.486$ & $92$ & $6.029$ \\ \hline
$(4 + \sqrt{3})$ & $13$ & $45 + 27 \sqrt{3}$ & $9.038$ & $92$ & $6.029$ \\ \hline
$(2 - 3 \sqrt{3})$ & $23$ & $71 + 40 \sqrt{3}$ & $9.887$ & $254$ & $7.383$ \\ \hline
$(2 + 3 \sqrt{3})$ & $23$ & $96 + 55 \sqrt{3}$ & $10.507$ & $254$ & $7.383$ \\ \hline
$(5)$ & $25$ & $(168 + 94\sqrt{3})\cos \frac{\pi}{12}$ & $11.534$ & $326$ & $7.716$ \\ \hline
$(7 + 2 \sqrt{3})$ & $37$ & $204 + 117 \sqrt{3}$ & $12.016$ & $2110$ & $10.206$ \\ \hline
$(7 - 2 \sqrt{3})$ & $37$ & $324 + 187 \sqrt{3}$ & $12.947$ & $2110$ & $10.206$ \\ \hline
$(1 + 4\sqrt{3})$ & $47$ & $(282 + 164 \sqrt{3})\cos \frac{\pi}{12}$ & $12.608$ & $2163$ & $10.239$ \\ \hline
$(1 - 4 \sqrt{3})$ & $47$ & $(378 + 222 \sqrt{3}) \cos \frac{\pi}{12}$ & $13.204$ & $2163$ & $10.239$ \\ \hline
$(7)$ & $49$ & $341 + 196 \sqrt{3}$ & $13.046$ & $2451$ & $10.406$ \\ \hline
$(4 - 5\sqrt{3})$ & $59$ & $564 + 325 \sqrt{3}$ & $14.054$ & $8556$ & $12.073$ \\ \hline
$(4 + 5\sqrt{3})$ & $59$ & $1115 + 644 \sqrt{3}$ & $15.420$ & $8556$ & $12.073$ \\ \hline
\end{longtable}
\end{center}

\subsection{Bolza twins} \label{sec:bolza}
Consider $\tau = (2,3,8)$.  Then $F_\tau = \Q \left( \cos \frac{\pi}{8} \right)$ and $E_\tau = \Q(\sqrt{2})$.
From Proposition~\ref{pro:presentation.a} we find that $A_\tau = \left\langle \frac{-3, \sqrt{2}}{\Q(\sqrt{2})} \right\rangle$ and that the order $\mathcal{Q}_\tau \subset A_\tau$ is spanned over $\Z[\sqrt{2}]$ by $\{1, \gamma_1, \gamma_2, \gamma_1 \gamma_2 \}$, where
$\gamma_1 = \frac{1}{2} (-1 + i)$ and $\gamma_2 = \frac{\sqrt{2}}{2} + \frac{2 + \sqrt{2}}{6} i + \frac{1}{3} ij$.  It is easy to see that $\mathcal{Q}_\tau$ is the ``Bolza order'' defined in~\cite[\S7]{KKSV/16}; this follows from noting that $\gamma_1 = \alpha - 1$ and $\gamma_2 = (1 + \sqrt{2})\alpha - \beta$, where $\alpha$ and $\beta$ are as in~\cite{KKSV/16}.
Moreover $\mathcal{Q}_\tau^1 / \{ \pm 1 \} \simeq \overline{\Delta}^{(2)}(\tau) \simeq \overline{\Delta}(3,3,4)$~\cite[\S8]{KKSV/16}.  Katz, Katz, the first author, and Vishne~\cite[\S14-16]{KKSV/16} previously computed bounds on $\mathrm{sys}(X_{\overline{\Delta}(\tau^{(2)};\p)})$, using the method of Section~\ref{sec:procedure}, for prime ideals $\p \vartriangleleft \Z[\sqrt{2}]$ dividing a rational prime $p$ that splits in $\Q(\sqrt{2})$.
It is an exercise in algebraic number theory to deduce the following from Proposition~\ref{pro:CV.congruence.quotient}.

\begin{pro} \label{pro:bolza}
Let $\tau = (2,3,8)$ and $\tau^{(2)} = (3,3,4)$.  Let $\p$ be a prime ideal of $\Z[\sqrt{2}]$ dividing the rational prime $p \geq 5$.  Then
\begin{eqnarray*}
\overline{\Delta}(\tau)/\overline{\Delta}(\tau;\p) & \simeq & \begin{cases}
\mathrm{PSL}_2(\F_p) &: p \equiv \pm 1, \, \mathrm{mod} \, 16 \\
\mathrm{PGL}_2(\F_p) &: p \equiv \pm 7, \, \mathrm{mod} \, 16 \\
\mathrm{PGL}_2(\F_{p^2}) &: p \equiv \pm 3, \pm 5 \, \mathrm{mod} \, 16
\end{cases} \\
\overline{\Delta}(\tau^{(2)}) / \overline{\Delta}(\tau^{(2)};\p) & \simeq & \begin{cases}
\mathrm{PSL}_2(\F_p) &: p \equiv \pm 1, \mathrm{mod} \, 8 \\
\mathrm{PSL}_2(\F_{p^2}) &: p \equiv \pm 3, \mathrm{mod} \, 8.
\end{cases}
\end{eqnarray*}
\end{pro}

Let $\p \vartriangleleft \Z[\sqrt{2}]$ be a prime ideal dividing $p \geq 5$.
Since $\overline{\Delta}(\tau^{(2)}; \p) = \overline{\Delta}(\tau;\p) \cap \overline{\Delta}^{(2)}(\tau)$, it follows by Proposition~\ref{pro:bolza} that $\overline{\Delta}(\tau^{(2)}; \p) = \overline{\Delta}(\tau;\p)$ exactly when $p \not\equiv \pm 1 \, \mathrm{mod} \, 16$.  Thus for half of the primes $p$ splitting in $\Q(\sqrt{2})$ we have $X_{\overline{\Delta}(\tau;\p)} = X_{\overline{\Delta}(\tau^{(2)};\p)}$, whereas for half of them $X_{\overline{\Delta}(\tau^{(2)};\p)}$ is a non-trivial cover of $X_{\overline{\Delta}(\tau;\p)}$ and may have larger systolic length.

The prime $\p_2 = (\sqrt{2})$ does not fit into the framework of this paper.  However, analogously to the previous section, it gives rise to a normal subgroup $H_2 \trianglelefteq \overline{\Delta}(\tau^{(2)})$ of index $24$ such that $\overline{\Delta}(\tau^{(2)}) / H_2 \simeq \mathrm{SL}_2(\F_3)$ and such that $H_2$ is isomorphic to a surface group of genus two.  The Riemann surface $X_{H_2}$ is the Bolza surface, which is known to be maximal.  The four generators in our presentation of $H_2$ all correspond to elements of $\mathcal{Q}_\tau^1 / \{ \pm 1 \}$ of trace $\pm 2 (1 + \sqrt{2})$, which give the true systolic length of the Bolza surface~\cite[\S5]{Schmutz/93}.  We observe in passing two other interesting normal subgroups $H_3, H_5 \trianglelefteq \overline{\Delta}(\tau^{(2)})$ of indices $48$ and $96$.  They are isomorphic to surface groups of genera $3$ and $5$, respectively, and neither appears to be a congruence subgroup.  We have $\overline{\Delta}(\tau^{(2)})/H_3 \simeq (\Z / 4\Z)^2 \rtimes \Z / 3\Z$, with $\Z / 3\Z$ acting faithfully, and the generators of $H_3$ correspond to traces $\pm 2 (2 + \sqrt{2})$.  Finally, $\overline{\Delta}(\tau^{(2)})/H_5 \simeq (\Z / 2\Z)^3 \rtimes A_4$, with $A_4$ acting via its quotient $\Z / 3\Z$.  The generators of $H_5$ correspond to traces $\pm 2(3 + 2 \sqrt{2})$.  Note that Schmutz Schaller~\cite[\S7]{Schmutz/93} discusses a maximal surface of genus $5$ and systolic length $2 \arcosh (3 + 2 \sqrt{2})$.

\subsection{$(2,7,7)$ triangle surfaces} \label{sec:277}
The case of $\tau = (2,7,7)$ is interesting for at least two reasons.  First of all, the analogue of Propositions~\ref{pro:hurwitz.normal.subgroups} and~\ref{pro:2312.normal.subgroups} fails.  Secondly, $\overline{\Delta}(2,7,7)$ is not maximal among triangle groups, since it embeds in $\overline{\Delta}(2,3,7)$ as a self-normalizing subgroup of index nine~\cite[Theorem~2]{Singerman/72}.  An explicit embedding is given by
\begin{eqnarray*}
f: \overline{\Delta}(2,7,7) = \langle x^\prime ,y^\prime \rangle / \langle (x^\prime)^2, (y^\prime)^7, (x^\prime y^\prime)^7 = 1 \rangle & \to & \langle x, y \rangle / \langle x^2, y^3, (xy)^7 \rangle = \overline{\Delta}(2,3,7)\\
x^\prime & \mapsto & x \\
y^\prime & \mapsto & (xy)^4 y,
\end{eqnarray*}
and its image is unique up to conjugation.
In this situation, $E_\tau = F_\tau = \Q(\cos \frac{\pi}{7})$ is the same field as in Section~\ref{sec:hurwitz}.  Recall that we set $\mu = 2 \cos \frac{\pi}{7}$.  The presentation of the quaternion algebra $B_\tau = A_\tau$ given by Proposition~\ref{pro:alpha.beta.fulldelta} is $\left \langle  \frac{-4, 2 \mu^2 - 4}{\Q(\mu)} \right\rangle$.  This is actually the same algebra as in Section~\ref{sec:hurwitz}, as evidenced by the isomorphism
\begin{eqnarray*}
\varphi: \left \langle  \frac{-4, 2 \mu^2 - 4}{\Q(\mu)} \right\rangle & \to & \left \langle  \frac{-4, \mu^2 - 3}{\Q(\mu)} \right\rangle \\
\varphi(i) & = & i \\
\varphi(j) & = & (\mu^2 - 1) \left( -j + \frac{ij}{2} \right).
\end{eqnarray*}
We work with the presentation of Section~\ref{sec:hurwitz}.  The (image under $\varphi$ of) the order $\mathcal{O}_{(2,7,7)}$ is the $\Z[\mu]$-span of $\{1, \alpha, \beta^\prime, \alpha \beta^\prime \}$, for $\alpha = \frac{i}{2}$ and $\beta^\prime = \frac{1}{4} \left( 2 \mu + \mu i + 2(\mu^2 - 1)j + (\mu^2 - 1)ij \right)$.  It is easy to check that $\mathcal{O}_{(2,7,7)} \subset \mathcal{O}_{(2,3,7)}$ and that $\mathcal{O}_{(2,7,7)}$ has discriminant $(8)$.  We have $\mathcal{O}^1_{(2,7,7)} / \{ \pm 1 \} \simeq \overline{\Delta}(2,7,7)$.  If $\tau$ is either $(2,3,7)$ or $(2,7,7)$, it follows from the construction in~\cite[\S5]{CV/19} that $\overline{\Delta}(\tau;\p) = \{ g \in \mathcal{O}_\tau^1 : g \equiv \pm 1 \, \mathrm{mod} \, \p \mathcal{O}_\tau \} / \{ \pm 1 \}$.  Hence
$$\overline{\Delta}(2,7,7;\p) = \overline{\Delta}(2,3,7;\p) \cap \overline{\Delta}(2,7,7)$$
for all prime ideals $\p \vartriangleleft \Z[\mu]$ not dividing $2$.  In this case, $X_{\overline{\Delta}(2,7,7;\p)}$ is a cover of degree nine of $X_{\overline{\Delta}(2,3,7,\p)}$, and thus the length spectrum of $X_{\overline{\Delta}(2,7,7;\p)}$ is a subset of that of $X_{\overline{\Delta}(2,3,7;\p)}$.  The following table collects our estimates for $\mathrm{sys}(X_{\overline{\Delta}(2,7,7;\p)})$ for some primes $\p$.  The last column gives the rank of our candidate for $\mathrm{sys}(X_{\overline{\Delta}(2,7,7;\p)})$ in the length spectrum of $X_{\overline{\Delta}(2,3,7;\p)}$, as computed by Vogeler~\cite{Vogeler/03}.  Where the entry in this column is $1$, we have certainly obtained the true value of $\mathrm{sys}(X_{\overline{\Delta}(2,7,7;\p)})$.

\begin{center}
\begin{longtable}{| r | r || r | r || r | r |} \hline
$I$ & $N(I)$ & $x =$ lowest observed trace & $2 \arcosh \frac{x}{2}$ & $g(X_{\overline{\Delta}(2,7,7;I)})$ &  \\ \hline \hline
$(\mu + 2)$ & $7$ & $ 9 \mu^2 + 8 \mu - 4$ & $7.358$ & $19$ & $3$ \\ \hline
$(2)$ & $8$ & $4 \mu^2 + 4 \mu - 2$ & $5.796$ & $49$ & $1$ \\ \hline
$(2 \mu + 1)$ & $13$ & $17 \mu^2 + 12 \mu - 10$ & $8.404$ & $118$ & $2$ \\ \hline
$(2 \mu + 3)$ & $13$ & $6 \mu^2 + 5 \mu - 4$ & $6.393$ & $118$ & $1$ \\ \hline
$(\mu - 3)$ & $13$ & $8 \mu^2 + 7 \mu - 4$ & $7.085$ & $118$ & $2$ \\ \hline
$(3)$ & $27$ & $135 \mu^2 + 108 \mu - 74$ & $12.652$ & $1054$ & $3$ \\ \hline
$(4 \mu - 1)$ & $29$ & $223 \mu^2 + 180 \mu - 124$ & $13.658$ & $1306$ & $4$ \\ \hline
$(3 \mu - 4)$ & $29$ & $49 \mu^2 + 41 \mu - 27$ & $10.656$ & $1306$ & $1$ \\ \hline
$(\mu + 3)$ & $29$ & $73 \mu^2 + 60 \mu - 40$ & $11.442$ & $1306$ & $1$ \\ \hline
\end{longtable}
\end{center}

The case of $\p = (2)$, which divides the discriminant of $\mathcal{O}_{(2,7,7)}$, is exceptional.  The congruence subgroup $\overline{\Delta}(2,3,7;(2))$ corresponding to the Fricke-Macbeath curve of genus $7$ is contained in (any embedding of) $\overline{\Delta}(2,7,7)$ into $\overline{\Delta}(2,3,7)$.  Thus the Fricke-Macbeath curve is a $(2,7,7)$-triangle surface.  However, 
$$\overline{\Delta}(2,7,7;(2)) = \{ g \in \mathcal{O}^1_{(2,7,7)} : g \equiv \pm 1 \, \mathrm{mod} \, 2 \mathcal{O}_{(2,7,7)} \}$$ 
is a subgroup of index eight in $\overline{\Delta}(2,3,7;(2))$.  Thus $X_{\overline{\Delta}(2,7,7;(2))}$ is a cover of degree $8$ of the Fricke-Macbeath curve, although it has the same systolic length.

It follows from Proposition~\ref{pro:counting.normal.subgroups}, by an argument that we have already seen in several previous examples, that $| \mathcal{N}(\tau;p) | \leq 9$ for all primes $p \not\in \{ 2, 7 \}$.  The smallest primes that split in $\Q(\mu)$ are $13$ and $29$.  For either $p \in \{13, 29 \}$, we find nine normal subgroups $H \trianglelefteq \overline{\Delta}(2,7,7)$ such that $\overline{\Delta}(2,7,7)/H \simeq \mathrm{PSL}_2(\F_p)$, corresponding to the nine triples $\underline{C} = (C_1, C_2, C_3)$ of conjugacy classes of $\mathrm{PSL}_2(\F_p)$ with $o(\underline{C}) = (2,7,7)$; note that there is only one possibility for $C_1$, but three for $C_2$ and $C_3$.  In each case, the three subgroups associated to triples $\underline{C}$ with $C_2 = C_3$ are the congruence subgroups arising from prime ideals dividing $p$.  The six remaining normal subgroups in each case are also isomorphic to surface groups of genera $118$ or $1306$.  They are not all conjugate in $\mathrm{PSL}_2(\R)$; note that Theorems~5 and~6 of~\cite{GW/05} tell us where to look for conjugating elements.  However, for all twelve of these groups, the lowest observed trace is $11 \mu^2 + 9 \mu - 7$, corresponding to a systolic length of $7.609\dots$  This is precisely $\mathrm{sys}(X_S)$, where $S \trianglelefteq \overline{\Delta}(2,3,7)$ is either of Sinkov's non-congruence normal subgroups of index $1344$.  Similarly, there are three normal subgroups of $\overline{\Delta}(2,7,7)$ with quotient isomorphic to $\mathrm{PSL}_2(\F_{27})$.  One of them is the congruence subgroup $\overline{\Delta}(2,7,7;(3))$.  The other two are surface groups of genus $1054$, and the lowest observed trace is again $11 \mu^2 + 9 \mu - 7$.  There are two non-congruence normal subgroups $H \trianglelefteq \overline{\Delta}(2,7,7)$ with $\overline{\Delta}(2,7,7)/H \simeq \mathrm{PSL}(\F_8)$, whose lowest observed trace is the same mysterious $11 \mu^2 + 9 \mu - 7$.  While the pairwise intersections of all these groups with the same lowest observed trace are small, a very disproportionate number of the generators we use lie in these intersections.  This may account for our computations.  We hope to return to this interesting phenomenon in future work.

\subsection{$(3,3,10)$ triangle surfaces}
Finally, consider the triple $\tau = (3,3,10)$.  By~\cite[Theorem~3]{Takeuchi/77}, the triangle group $\overline{\Delta} = \overline{\Delta}(\tau)$ is not arithmetic.  We have $F_\tau = E_\tau = \Q (\nu)$, where $\nu = 2 \cos \frac{\pi}{10}$ has minimal polynomial $x^4 - 5x^2 + 5$ over $\Q$.  This number field has class number $1$.  The quaternion algebra $B_\tau = A_\tau = \left\langle \frac{-3, \nu^2 + \nu - 2}{F_\tau} \right\rangle$ splits at two of the four infinite places of $F_\tau$ and by Proposition~\ref{pro:alpha.beta.fulldelta} the order $\mathcal{O}_\tau \subset B_\tau$ is spanned over $\mathcal{O}_{F_\tau}$ by $\{ 1, \alpha, \beta, \alpha \beta \}$, where $\alpha = \frac{1}{2} + \frac{i}{2}$ and $\beta = \frac{1}{2} + \frac{1 + 2 \nu}{6} i - \frac{1}{3} ij$.  By Theorem~\ref{thm:semiarithmetic} we have $\mathrm{sys}(X_{\overline{\Delta}(I)}) > \frac{2}{3} \log g(X_{\overline{\Delta}(I)}) - c$ for all ideals $I \triangleleft \mathcal{O}_{F_\tau}$ coprime to $30$.

\begin{lemma}
Let $\tau = (3,3,10)$ and let $p \geq 7$ be a rational prime.  Then
$$ K(\tau;p) \simeq \begin{cases}
\mathrm{PSL}_2(\F_p) &: p \equiv 1, 19 \, \mathrm{mod} \, 20 \\
\mathrm{PSL}_2(\F_{p^2}) &: p \equiv 9, 11 \, \mathrm{mod} \, 20 \\
\mathrm{PSL}_2(\F_{p^4}) &: p \equiv 3, 7, 13, 17 \, \mathrm{mod} \, 20.
\end{cases}$$
\end{lemma}
\begin{proof}
The claim follows from Proposition~\ref{pro:CV.congruence.quotient} and reasoning similar to the proof of Lemma~\ref{lem:hurwitz.decomposition}.
\end{proof}

We tabulate our bounds on $\mathrm{sys}(X_{\overline{\Delta}(I)})$ for some prime ideals of $\mathcal{O}_{F_\tau}$.  It is evident from the table that there exist ideals for which $\mathrm{sys}(X_{\overline{\Delta}(I)}) < \frac{4}{3} \log g(X_{\overline{\Delta}(I)})$.
\begin{center}
\begin{longtable}{| r | r || r | r || r | r |} \hline
$\p$ & $N(\p)$ & $x =$ lowest observed trace & $2 \arcosh \frac{x}{2}$ & $g(X_{\overline{\Delta}(\p)})$ & $\frac{2}{3} \log g$ \\ \hline \hline
$(\nu^2 + \nu - 4)$ & $19$ & $21 + 12 \nu - 18 \nu^2 - 10 \nu^3$ & $9.002$ & $400$ & $3.994$ \\ \hline
$(\nu^2 - \nu - 4)$ & $19$ & $24 + 10 \nu - 20 \nu^2 - 10 \nu^3$ & $9.173$ & $400$ & $3.994$ \\ \hline
$(\nu^3 - \nu^2 - 3 \nu + 1)$ & $19$ & $40 + 18 \nu - 32 \nu^2 - 16 \nu^3$ & $10.043$ & $400$ & $3.994$ \\ \hline
$(\nu^3 + \nu^2 - 3 \nu - 1)$ & $19$ & $86 + 46 \nu - 64 \nu^2 - 34 \nu^3$ & $11.354$ & $400$ & $3.994$ \\ \hline
$(\nu^3 - 3 \nu - 3)$ & $41$ & $6 + 4 \nu - 7 \nu^2 - 4 \nu^3$ & $7.338$ & $4019$ & $5.533$ \\ \hline
$(\nu - 3)$ & $41$ & $36 + 18 \nu - 27 \nu^2 - 14 \nu^3$ & $9.637$ & $4019$ & $5.533$ \\ \hline
$(\nu^3 - 3 \nu + 3)$ & $41$ & $43 + 20 \nu - 32 \nu^2 - 16 \nu^3$ & $9.951$ & $4019$ & $5.533$ \\ \hline
$(\nu + 3)$ & $41$ & $243 + 125 \nu - 178 \nu^2 - 93 \nu^3$ & $13.377$ & $4019$ & $5.533$ \\ \hline
$(\nu^3 + \nu^2 - 2\nu - 4)$ & $59$ & $74 + 36\nu - 57 \nu^2 - 29 \nu^3$ & $11.147$ & $11978$ & $6.261$ \\ \hline
$(\nu^3 - \nu^2 - 4 \nu + 1)$ & $59$ & $90 + 45\nu - 66 \nu^2 - 34 \nu^3$ & $11.389$ & $11978$ & $6.261$ \\ \hline
$(\nu^3 + \nu^2 - 4\nu - 1)$ & $59$ & $301 + 156 \nu - 218 \nu^2 - 114 \nu^3$ & $13.766$ & $11978$ & $6.261$ \\ \hline
$(\nu^3 - \nu^2 - 2\nu + 4)$ & $59$ & $383 + 200 \nu - 278 \nu^2 - 146 \nu^3$ & $14.257$ & $11978$ & $6.261$ \\ \hline
$(3)$ & $81$ & $259 + 135 \nu - 189 \nu^2 - 99 \nu^3$ & $13.489$ & $30997$ & $6.894$ \\ \hline
\end{longtable}
\end{center}

\begin{acknowledgements}
We are grateful to Mikhail Katz and John Voight for helpful conversations and correspondence and to the anonymous referee for comments that improved the exposition.
\end{acknowledgements}

\bibliographystyle{amsplain}

\end{document}